\documentclass[12pt]{amsart}
\input xy
\xyoption{all}
\usepackage[english]{babel}
\usepackage[latin1]{inputenc}
\usepackage{graphicx}
\usepackage{amsmath}
\usepackage{amssymb}
\usepackage{amscd}
\usepackage{color}
\usepackage{oldgerm}
\usepackage{amsfonts}
\usepackage{newlfont}
\usepackage{longtable}
\usepackage{multirow}

\usepackage{fontenc}

\DeclareOldFontCommand{\rm}{\normalfont\rmfamily}{\mathrm}

\usepackage[all]{xy}

\usepackage{upref}

\def\F{\Bbb F}

\def\ad{\operatorname{ad}}

\def\Der{\operatorname{Der}}

\def\dim{\operatorname{dim}}

\def\End{\operatorname{End}}

\def\Tr{\operatorname{Tr}}
\def\Hom{\operatorname{Hom}}
\def\Ker{\operatorname{Ker}}

\def\Im{\operatorname{Im}}

\def\span{\operatorname{Span}}
\def\Nil{\operatorname{Nil}}

\def\g{\frak g}

\def\a{\frak{a}}

\theoremstyle{plain}\swapnumbers

\newtheorem{Theorem}{Theorem}[section]

\newtheorem{Lemma}[Theorem]{Lemma}
\newtheorem{Prop}[Theorem]{Proposition}

\newtheorem{Cor}[Theorem]{Corollary}

\newtheorem{Remark}[Theorem]{Remark}

\title[Hom-Lie algebra structures on quadratic Lie algebras]
{Hom-Lie algebra structures on quadratic Lie algebras
and twisted invariant Killing-like forms defined on them}

\author{R. Garc\'ia-Delgado, G. Salgado, O. A. S\'anchez-Valenzuela}

\address{CIMAT-Unidad M\'erida}

\email{rosendo.garcia@cimat.mx}
\email{gsalgado@fciencias.uaslp.mx, gil.salgado@gmail.com}
\email{adolfo@cimat.mx}


\date{\today}

\begin{document}

\maketitle

\begin{abstract}
Hom-Lie algebras defined on
central extensions
of a given quadratic Lie algebra
that in turn admit an invariant metric,
are studied.
It is shown how some of these algebras are naturally equipped with
other symmetric, bilinear forms that satisfy an invariant condition
for their twisted multiplication maps.
The twisted invariant bilinear forms
so obtained resemble 
the Cartan-Killing forms defined
on ordinary Lie algebras. 
This fact allows one to reproduce on 
the Hom-Lie algebras hereby studied, some results that are 
classically associated to the ordinary Cartan-Killing form.
\end{abstract}

\section*{Introduction}

The algebraic study of Hom-Lie algebras can be traced back to the
work of Hartwig,  Larsson  and  Silvestrov (see \cite{HLS}).
A Hom-Lie algebra is a triple $(\g,\mu,\alpha)$ 
consisting of a vector space $\g$ (over a fixed ground field $\Bbb F$
which will be assumed to be 
of characteristic zero),
a skew-symmetric bilinear map $\mu:\g\times\g\to\g$
and an $\Bbb F$-linear endomorphism $\alpha:\g\to\g$ (also called
{\it the twist map\/} of the algebra) that satisfies the following
{\it twisted Jacobi identity\/:}
$$
\displaystyle{\sum_{\circlearrowleft\{x,y,z\}}}\!\!\!\mu(\alpha(x),\mu(y,z))=0.
$$
When this identity is satisfied for $\alpha=\operatorname{Id}_\g$,
the triple $(\g,\mu,\alpha)$ is a Lie algebra with Lie bracket 
$[\,\cdot\,,\,\cdot\,]=\mu(\,\cdot\,,\,\cdot\,)$.
Thus, one may say that Lie algebras form a subclass of Hom-Lie algebras.  
Other subclass is obtained by restricting $\alpha$ to be
a homomorphism for the product $\mu$;
one may refer to the Hom-Lie algebras thus restricted as 
{\it multiplicative\/.} These can be further limited by requiring
$\alpha$ to be an invertible homomorphism; that is, an automorphism in the
category of pairs $(\g,\mu)$ giving
rise to the nowadays called {\it regular Hom-Lie algebras\/.}
Special families of these restrictions have been studied
by starting with an ordinary Lie algebra $(\g,[\,\cdot\,,\,\cdot\,])$
and looking for the conditions required to satisfy the twisted Jacobi identity;
either when $\alpha$ is a Lie algebra endomorphism or 
when it is a Lie algebra automorphism.
Not so much work has been done for the general case in which 
$\mu$ is not a Lie bracket in
$\g$ nor for the general triples
$(\g,\mu, \alpha)$
for which $\alpha$ is not necessarily a morphism for 
the {\it twisted product\/} $\mu$
---as one usually refers to $\mu$ when it is not a Lie bracket.

\medskip
Some of the classical concepts and constructions for Lie algebras
can be generalized to the realm of Hom-Lie algebras; {\it eg\/,} the 
concepts of Hom-Lie subalgebra, Hom-Lie ideal or Hom-Lie morphism
are defined in the obvious standard way.
The concept of {\it representation of a Hom-Lie algebra\/}
leads one to look at
the linear maps $x\mapsto\operatorname{ad}_{\mu}(x)=\mu(x,\,\cdot\,):\g\to\g$ 
under the context of the possible Hom-Lie algebra structures
that might be defined on $\operatorname{End}_{\Bbb F}(\g)$.

\medskip
There are some available works dealing with the representation theory
of Hom-Lie algebras under various restrictions (see for example
\cite{ABM} or \cite{Sheng}).
There are also some available works studying invariant quadratic forms
on $\g$ under the hypotheses that
$\g$ is a Lie algebra with Lie bracket $[\,\cdot\,,\,\cdot\,]$
and $\alpha$ is
a Lie algebra automorphism that satisfies the
twisted Jacobi identity.

\medskip
Motivated by the representation theory of Hom-Lie algebras on the one hand
and by the geometric condition arising from some special invariant
quadratic forms on them on the other, we
were led to some results that are somehow reminiscent of the
Cartan-Killing form of a Lie algebra. 
Namely, for a Hom-Lie algebra $(\g,\mu,\alpha)$
we consider the symmetric bilinear form 
$K:\g\times\g\to\Bbb F$ given by
$$
K(x,y) = \operatorname{Tr}(\operatorname{ad}_{\mu}(x)
\circ\operatorname{ad}_{\mu}(y)\circ\alpha).
$$
Under certain conditions 
$\alpha\circ\mu\,\left(\,\cdot\,,\,\cdot\,\right)$ is a Lie bracket
(see {\bf Thm. \ref{Teorema Cartan-Killing}})
and $K$ satisfies the invariance property,
$$
K\left(\operatorname{ad}_{\mu}(x)(y),z\right)+
K\left(y,\operatorname{ad}_{\mu}(x)(z)\right) = 0.
$$
Write $[\,\cdot\,,\,\cdot\,]$ instead of
$\alpha\circ\mu\,\left(\,\cdot\,,\,\cdot\,\right)$ when the latter is a Lie bracket.
It turns out
that if $K(x,[y,z])=0$ for any $x,y,z\in \g$, then $\g$ is a solvable Lie algebra.
Also, $K=0$, if and only if $\g$ is a nilpotent Lie algebra. 
It can be proved, however, that if the Hom-Lie algebra $(\g,\mu,\alpha)$ is nilpotent
then $K=0$ (see {\bf Prop. \ref{GilStatement}}), 
but the converse in the Hom-Lie category is not true:
If $K=0$, the Hom-Lie algebra $(\g,\mu,\alpha)$ is not necessarily nilpotent
(see \S 5 below).
On the other extreme, $K$ is non-degenerate if and only if the Cartan-Killing
form $\kappa$ of $(\g,[\,\cdot\,,\,\cdot\,])$ is non-degenerate, which implies that $\g$ is
a semisimple Lie algebra.
As far as we know, no previous
steps had been taken in this direction and we have found them interesting 
and appealing as $K$ gives a way to relate the Hom-Lie algebra $(\g,\mu,\alpha)$
with the Lie algebra $(\g,[\,\cdot\,,\,\cdot\,])$
(see {\bf Thm. \ref{Teorema Cartan-Killing}}).

\medskip
The origin of this work is \cite{GSV}, where the following
problem is addressed and solved:
{\sl Given a quadratic Lie algebra $\g$
having a central ideal $V$ such that the quotient $\g_0=\g/V$
has the structure of a quadratic
Lie algebra whose invariant metric $B_0:\g_0\times\g_0\to\Bbb F$
is not given by the restriction of the invariant bilinear form 
$B:\g\times\g\to\Bbb F$, can the structure of $\g$ be recovered from that of $\g_0$?}
We have shown in \cite{GSV} how to reconstruct $\g$ from $\g_0$ and $V$ and
how to produce an invariant metric $B$ on it, given the fact
that the extension
$V\hookrightarrow\g\to\g/V\simeq\g_0$ is central.

\medskip
The aim of this work is to show that
a quadratic Lie algebra $\g_0$ having
a central extension $\g=\g_0\oplus V$ equipped with an
invariant metric $B$, 
can also be equipped with skew-symmetric 
bilinear maps  $\mu_0$ and linear
endomorphisms $\alpha_0:\g_0\to\g_0$ making the triples
$(\g_0,\mu_0,\alpha_0)$ into Hom-Lie algebras
which also admit bilinear forms 
$K_0:\g_0\times\g_0\to\Bbb F$
of the Cartan-Killing type described above, 
satisfying,
$$
K_0(\mu_0(x,y),z)=K(x,\mu_0(y,z)),\quad\text{for all\ }x,y,z \in \g_0.
$$
We shall show how to produce two different 
Hom-Lie algebra structures on a central extension 
$\g=\g_0\oplus V$
of a Lie algebra $\g_0$ by a central ideal $V$
---say, $(\g,\mu,\alpha)$ and $(\g,\mu,\alpha^\prime)$, repectively---
one of which occurs when the corresponding Lie-algebra sequence
$V\hookrightarrow\g\twoheadrightarrow\g_0$ does not split.
For the Hom-Lie algebra produced in the non-split case,
we show how to associate
a linear  (algebraic) connection ---which we write as a 
bilinear {\it multiplication\/} map $(x,y)\mapsto xy\,$--- adapted to 
the twisted product $\mu$ (see {\bf Thm. \ref{teorema}}).
In this case,
we show that when the quadratic central extension
is produced in such a way that $V$ is isotropic,
this multiplication map on $\g$,
is neither symmetric ({\it ie\/,} commutative) nor skew-symmetric
(see \textbf{Lemma \ref{invariancia hom-lie}}).

\vfill

\medskip
\section{Central extensions of Lie algebras;\\
Notation and conventions}

\medskip
All Lie algebras in this paper are finite dimensional
over an algebraically closed field $\F$ of characteristic zero.
We shall closely follow the approach of \cite{GSV} with a slight
change of notation. For the sake of being self-contained and for
our readers' benefit, we reproduce here the setting used in \cite{GSV}
and quote the main results there on which we now base this work.

\medskip
Let $\g$ be a Lie algebra and let $[\,\cdot\,,\,\cdot\,]:\g\times\g\to\g$
be its Lie bracket. 
The Lie algebra $\g$ {\bf is} said to be {\bf quadratic} if it comes equipped
with a non-degenerate, symmetric, bilinear form, $B:\g\times\g\to\Bbb F$, satisfying,
$B([x,y],z)=B(x,[y,z])$, for any $x$, $y$ and $z$ in $\g$.
The bilinear form $B$ {\bf is} said to be {\bf an invariant metric\/} 
on $\g$ if this property is satisfied,
or say simply that $B$ {\bf is invariant}.

\medskip
For a given non-degenerate, symmetric, bilinear form 
$B:\g\times\g\to\Bbb F$ on a vector space $\g$,
let $B^{\flat}:\g\to\g^\ast=\Hom(\g,\F)$ be the map defined by
$B^{\flat}(x)\,(z)=B(x,z)$, for all $z \in\g$, and each $x\in\g$.
Since $\dim\g$ is assumed to be finite, one may define the inverse map,
$B^\sharp:\g^*\to\g$, satisfying the condition,
$B(B^\sharp(\theta),y)=\theta(y)$, for any $\theta\in\g^*$ and any $y\in\g$. 
When $\g$ is a Lie algebra and $B$ is an invariant metric on it,
$B^{\flat}:\g \to \g^\ast$ is actually an
isomorphism that intertwines the adjoint action of $\g$ in $\g$ with its
coadjoint action in $\g^\ast$;
that is,  $B^{\flat}(\ad(x)(y))\,(z) = (\ad^\ast\!(x) B^{\flat}(y))\,(z)$, 
where $\ad^\ast\!(x)(\theta)=-\theta\circ\ad(x)$, for any $\theta\in\frak{g}^\ast$.

\medskip
For any Lie algebra $(\g,[\,\cdot\,,\,\cdot\,])$, we shall denote by
$\Gamma(\g)$ the vector subspace of $\End(\g)$ consisting of those
linear maps that {\it commute with the adjoint action\/:}
$$
\Gamma(\g)=\{T \in \End(\g)\,\mid\,T\circ\ad(x)=\ad(T(x)),
\ \text{for all\ } x \in \g\, \}.
$$
We shall refer to the linear maps $T\in \Gamma(\g)$ as
{\bf $\ad$-equivariant maps}. They are also referred to
as {\bf centroids} in the literature (see  \cite{Zhu}).
If $B:\g\times\g\to\Bbb F$ is an invariant metric on $\g$,
we shall refer ourselves to the {\bf self-adjoint} linear operators in $\End(\g)$
as $B$-{\bf symmetric maps}.
We shall use the following notation
for $B$-{\bf symmetric $\ad$-equivariant maps}:
$$
\Gamma_{B}(\g)=\{\,T \in
\Gamma(\g)\,\mid\,B\left(T(x),y\right)=B\left(x,T(y)\right),\ \text{for all\ }x,y \in \g\,\}.
$$
It is easy to see that a bilinear form 
$\tilde{B}:\g\times\g\to\F$ yields another invariant metric on $\g$, 
if and only if there is
an {\it invertible map\/}
$T \in \Gamma_{B}(\g)$
such that $\tilde{B}(x,y)=B\left(T(x),y\right)$ for all $x,y \in \g$.

\medskip
Let $(\g_0,[\,\cdot\,,\,\cdot\,]_0)$ be a Lie algebra and let $V$ be a 
finite-dimensional vector space
on which $\g_0$ {\bf acts via the trivial representation.} 
A skew-symmetric bilinear map $\theta:\g_0\times\g_0\to\ V$ is a 
\textbf{2-cocycle of $\g_0$ with values in $V$}, if for any 
$x$, $y$ and $z$ in $\g_0$,
$$
\theta(x,[y,z]_0)+\theta(y,[z,x]_0)+\theta(z,[x,y]_0)=0.
$$
The vector space of $2$-cocycles of $\g_0$ with values in $V$
is usually denoted by $Z^2(\g_0,V)$. When the context makes it clear, 
one simply refers to $\theta$'s in $Z^2(\g_0,V)$
as {\bf $2$-cocycles}. A \textbf{$2$-coboundary} is a $2$-cocycle $\theta$ for which
there exists a linear map  $\tau:\g_0\to V$, satisfying
$\theta(x,y)=\tau([x,y]_0)$, for all $x,y \in \g_0$. 
It is well-known that when $\theta$
is a $2$-cocycle,
one may define a Lie bracket $[\,\cdot\,,\,\cdot\,]:\g\times\g\to\g$
(depending on $\theta$) in the vector space direct sum $\g = \g_0\oplus V$,
by means of,
\begin{equation}
\label{tc1}
[x+u,y+v] = [x,y]_0 + \theta(x,y),\quad\text{for all\ }x,y\in\g_0\ \ \text{and}\ \ u,v\in V.
\end{equation}
And conversely: if
\eqref{tc1} is a Lie bracket in $\g=\g_0\oplus V$,
then $\theta$ must be a $2$-cocycle. The resulting Lie algebra $\g$ is known as  
{\bf the central extension of $\g_0$ by $V$ associated to the cocycle $\theta$.}
Under these conditions, $V\subset\g$ {\bf is an ideal contained in the center of} $\g$
and there is a short exact sequence of Lie algebras
induced by the canonical projection $\pi:\g \to \g_0$, given by:
\begin{equation}
\label{m}
\xymatrix{ 
 0 \ar[r] &
V \ar[r]^{\iota} & 
{\frak{g}} \ar[r]^{\pi} &
{\frak{g}_0}\ar[r] &
 0 }
\end{equation}
The natural number $\dim V$
is referred to as
{\bf the dimension of the central extension.}
It is well known that the isomorphism class of the Lie algebra $\g$
is uniquely determined by the cohomology class $[\theta]$ of the $2$-cocycle
(see \cite{Che}). 
Now, the short exact sequence \eqref{m} {\bf splits}
if there exists a Lie algebra morphism, $\sigma:\g_0\to\g$
such that, $\pi \circ \sigma=\operatorname{Id}_{{\g}_0}$.
Equivalently, it splits if and only if there 
is a linear map $\tau:\g_0 \rightarrow V$, 
such that, $\theta(x,y)=\tau([x,y]_0)$, for all $x,y \in \g_0$.

\medskip
\noindent
Let $\g$ be a Lie algebra with Lie bracket $[\,\cdot\,,\,\cdot\,]$.
Let $\Der \g\subset \End(\g)$ be the subspace of
{\bf derivations of} $\g$:
$$
\Der\g = 
\{\,D\in\End(\g)\,\mid\, D([x,y])=[D(x),y]+[x,D(y)];\ x,y \in \g\,\}.
$$
This subspace is a Lie subalgebra
of $\End(\g)$ under the commutator of linear operators.
If $\g$ is quadratic with invariant metric $B$,
we shall say that a linear map $D \in \End(\g)$ {\bf is $B$-skew-symmetric} if it satisfies
$B\left(D(x),y\right)=-B\left(x,D(y)\right)$, for all $x$ and $y$ in $\g$. 
The vector subspace of $\End(\g)$ consisting of 
$B$-skew-symmetric linear maps is denoted by 
${\frak o}_{B}(\g)$; it is  also a Lie subalgebra of $\End(\g)$.
\smallskip

The following proposition, whose proof is straightforward,
makes it evident that there is a strong relationship
between $B$-skew-sym\-metric derivations
and $2$-cocycles:

\medskip
\begin{Prop}\label{proposicion 1}{\sl
Let $\g_0$ be a quadratic Lie algebra with Lie bracket $[\,\cdot\,,\,\cdot\,]_0$
and invariant metric $B_0$.
Let $V$ be an $r$-dimensional vector space. Each basis $\{v_1,\ldots,v_r\}$ of $V$
provides an isomorphism between the space of $2$-cocycles $Z^{2}(\g_0,V)$
and the space of $r$-tuples $(D_1,\ldots, D_r)$ of $B_0$-skew-symmetric derivations 
of $\g_0$, via,
$$
\qquad\theta(x,y)=
\sum_{i=1}^r B_0\left(D_i(x),y\right)\,v_i,\quad\text{for all\ }x,y\in\g_0.
$$
Moreover, a $2$-cocycle $\theta\in Z^{2}(\g_0,V)$ is a $2$-coboundary if and only if 
there are $a_i\in \g_0$ such that
$D_i=\ad_0(a_i)=[\,a_i\,,\,\cdot\,]_0$ (\,$1\le i\le r$).
}
\end{Prop}

\medskip
{\bf Hypotheses for the remaining part of this work.}
We shall assume that $V$ is an $r$-dimensional vector space
and that $\g=\g_0\oplus V$ is a central extension of $\g_0$
associated to a given $2$-cocycle $\theta\in Z^2(\g_0,V)$.
We shall assume as needed that a basis $\{v_1,\ldots,v_r\}$ 
of $V$ is given establishing the correspondence,
$$
Z^2(\g_0,V)\ni\theta
\ \leftrightarrow\  
(D_1,\ldots, D_r),
\quad 
D_i\in\Der\g_0\cap
{\frak o}_{B_0}(\g_0),
\ 1\le i\le r.
$$
We shall assume that the Lie bracket  $[\,\cdot\,,\,\cdot\,]$
in $\g$ is given by \eqref{tc1} and that
$\g_0$ and $\g$ are both quadratic with invariant metrics 
$B_0:\g_0\times\g_0\to\F$ and $B:\g\times\g\to\F$,
respectively. Let $\iota:\g_0\hookrightarrow\g$ and $\pi:\g\rightarrow\g_0$ be the inclusion
and projection maps in \eqref{m}. 
Define the linear maps,
\begin{equation}
\label{maps}
h=B_0^\sharp\circ\,\iota^*\,\circ\,B^\flat :\g\to\g_0,
\quad\text{and}\quad
k = B^\sharp\circ\,\pi^*\,\circ\,B_0^\flat:\g_0\to\g,
\end{equation}
with, $\iota^*\!:\g^*\!\to \g_0^*$ ($\iota^*(\xi)=\xi\circ\iota,\ \,\xi\in\g^*$)
and $\pi^*\!:\g_0^*\!\to \g^*$ ($\pi^*(\chi)=\chi\circ\pi,\ \,\chi\in\g_0^*$).
Observe that no {\it a priori\/} relationship between $B_0$ and $B$ is assumed
({\it eg\/,} $B_0$ is not necessarily the restriction of $B$ to $\g_0$)
and therefore $\g_0$ is neither a subalgebra nor an ideal of $\g$.
Nevertheless, the linear maps $h:\g \to \g_0$ and $k:\g_0 \to \g$, 
help each other to determine $B$ in terms of $B_0$ and viceversa,
as the following result from \cite{GSV} states:

\medskip
\begin{Lemma}\label{lema 1}{\sl
Under the stated hypotheses,
the linear maps $h:\g\to\g_0$ and $k:\g_0\to\g$ defined in \eqref{maps}, 
satisfy the following properties:
\begin{itemize}

\item[(i)] $h \circ k=\operatorname{Id}_{{\g}_0}$.

\item[(ii)] $B(x+v,y)=B_0(h(x+v),y)$ and 
$B_0(x,y)=B(k(x),y+v)$ for all $x,y \in \g_0$ and $v \in V$.

\item[(iii)] $\Ker(h)={\g_0}^{\perp}$ and\ \ $\Im(k)=V^{\perp}$.

\item[(iv)] $\g=\Ker(h) \oplus \Im(k)$.

\item[(v)] $k([x,y]_0)=[k(x),y]$, for all $x,y \in \g_0$.

\end{itemize}}
\end{Lemma}

Let $\pi_V:\g \to V$ be the linear projection onto $V$.
Set $\alpha_0=\pi \circ k:\g_0 \to \g_0$ 
and $R=\pi_{V} \circ k:\g_0 \to V$. 
Then $k(x)=\alpha_0(x)+R(x)$,
for all $x \in \g_0$. 
Having fixed a basis $\{v_1,\ldots,v_r\}$ of $V$, each $x \in \g_0$
gives rise to $r$ scalars $\alpha_\ell(x)\in \F$ 
($1\le \ell\le r$), such that $R(x)=\alpha_1(x)v_1+ \cdots +\alpha_r(x)v_r$.
Clearly, each assignment $x\mapsto\alpha_{\ell}(x)$ is linear.
Since $B_0$ is non-degenerate, for each $\alpha_\ell\in\g^*_0$,
there exists an element $a_{\ell} \in \g_0$, satisfying 
$\alpha_{\ell}(x)=B_0(a_{\ell},x)$,
for any $x \in \g_0$ ($1 \leq \ell \leq r$). 
Thus,
$$
R(x)=\sum_{i=1}^rB_0(a_i,x)\,v_i,\quad \text{for all\ }\,x \in \g_0,
$$ 
and,
\begin{equation}\label{k}
k(x)=\alpha_0(x)+\sum_{i=1}^rB_0(a_i,x)\,v_i,\quad \text{for all\ }\,x \in \g_0.
\end{equation}
Of prime 
motivation for this work
was to realize that
the linear map $h:\g \to \g_0$ gives rise to a 
linear map
$\rho:\g \to \Der(\g_0) \cap
{\frak o}_{B_0}(\g_0)
\subset\End(\g_0)$,
defined by,
\begin{equation}
\label{definicion rho}
\rho(x)(y)=h([x,y]), \quad \text{for all\ }\,x \in \g\ \ \text{and}\ \ y \in \g_0.
\end{equation}
It is shown in {\bf Lemma \ref{lema 2}.(iv)} below
that $\rho:\g\to\End(\g_0)$ behaves like a representation of $\g$ in $\g_0$. 
The fact that 
$\operatorname{Im}(\rho)\subset \Der(\g_0) \cap {\frak o}_{B_0}(\g_0)$
follows from item (ii) in {\bf Lemma \ref{lema 1}} and the
fact that $B$ and $B_0$ are non-degenerate.
But the important point for us is that $\rho$ somehow intertwines
the geometry and the algebra between $\g$ and $\g_0$. 
In order to explain how, we first quote the following result from \cite{GSV}
(see {\bf Lemma 2.2} therein)
that recalls some of the properties of $\rho$.

\medskip
\begin{Lemma}\label{lema 2}{\sl 
Under the stated hypotheses,
$\alpha_0=\pi\circ k\in \Gamma_{B_0}(\g_0)$
and there is a basis $\{a_i+w_i \mid\,a_i \in \g,\,w_i \in
V,\,\,1 \leq i \leq r \}$ of $(\g_0)^{\perp}$, such that:

\begin{itemize}

\item[(i)] $B(a_i+w_i,v_j)=\delta_{ij}$.
\smallskip

\item[(ii)] $\Ker(\alpha_0) \subseteq C(\g_0)$.
\smallskip

\item[(iii)] $ h([x,[y,z]])=[x,h([y,z])]_0$, 
for all $x \in \g_0$ and for all $y,z \in \g$.
\smallskip

\item[(iv)] $\rho([x,y])=\rho(x) \circ \alpha_0 \circ \rho(y)-\rho(y) \circ \alpha_0 \circ \rho(x)$,
for all $x,y \in \g_0$ and $\rho(a_i)=D_i$ for all $1 \leq i \leq r$.
\smallskip

\item[(v)] $\Ker(D_i)=\{x \in \g\,\mid\,[a_i,x]=0\}$,
for all $1 \leq i \leq r$.
\smallskip

\item[(vi)] $x=\alpha_0\circ h(x+v)+\overset{r}{\underset{i=1}{\sum}}B(x+v,v_i)a_i$, \ 
for all $x \in \g_0$ and $v \in V$.
\smallskip

\item[(vii)] $\alpha_0 \circ \rho(x)=\rho(x) \circ \alpha_0=\ad(x)$,
for all $x \in \g_0$. 
\newline
In particular, $\alpha_0 \circ D_i=D_i \circ \alpha_0=\ad(a_i)$, for all $1 \leq i \leq r$.

\end{itemize}
}
\end{Lemma}

Since most of our results deal with the case when $\g_0$
is a non-Abelian
Lie algebra, we single out the following result.

\medskip
\begin{Cor}\label{corolario zero}
{\sl Let $\alpha_0=\pi\circ k:\g_0 \to \g_0$
be the linear map defined in {\bf Lemma \ref{lema 2}}.
Under the stated hypotheses in this section,
$\operatorname{Im}(\alpha_0)=0$ implies that $\g_0$ is Abelian\/.}
\end{Cor}

\medskip
\begin{proof}
Since $V \subset C(\g)$, then $[k(x),y]=[\alpha_0(x),y]$, 
for all $x \in \g_0$ and for all $y \in \g$. 
If $\alpha_0=0$, {\bf Lemma \ref{lema 1}.(v)} implies that
$k([x,y]_0)=[k(x),y]=[\alpha_0(x),y]=0$, for all $x,y \in \g_0$. 
Since $k$ is injective ({\bf Lemma \ref{lema 1}.(i)}), 
$[x,y]_0=0$ for all $x,y \in \g_0$.
\end{proof}

\medskip
\section{Hom-Lie algebras defined on $\g=\g_0\oplus V$}

\medskip
\begin{Prop}\label{Prop 3}{\sl
Under the stated hypotheses in \S1 above,
there is a skew-symmetric, 
a bilinear product, 
$\mu:\g \times \g \to \g$,
and non-zero linear map, $\alpha:\g\to\g$  
such that:
\begin{itemize}
\item[(i)] $(\g,\mu,\alpha)$ is a Hom-Lie algebra, 
such that $\mu(\g,\g) \subset \g_0$, and $\mu(V,\g_0)=\{0\}$.
\end{itemize}
In particular,
$\mu(u,v)=0$, for all $u$ and $v$ in $V$, 
whereas the restriction 
$\mu_0=\mu\vert_{\g_0\times\g_0}:\g_0\times\g_0\to\g_0$
yields a skew-symmetric, 
bilinear product on $\g_0$,
such that,
if $\alpha_0=\pi\circ\alpha\vert_{\g_0}:\g_0\to\g_0$,
\begin{itemize}
\item[(ii)] 
the triple
$(\g_0,\mu_0,\alpha_0)$
is a Hom-Lie algebra, and the invariant metric $B_0$ of
the Lie algebra $\g_0$, satisfies, 
$$
\aligned
B_0(\mu_0(x,y),z) & = B_0(x,\mu_0(y,z)),\quad\text{and}\quad
\\
B_0\left(\alpha_0(x),y\right) & =B_0\left(x,\alpha_0(y)\right),
\quad\text{for all\ }x,y,z\in\g_0.
\endaligned
$$
\end{itemize}
\begin{itemize}
\item[(iii)] 
The triple $(V,\mu_V,\alpha_V)$,
with $\mu_V=\mu\vert_{V\times V}$ and $\alpha_V=\alpha\vert_V$,
is trivially a Hom-Lie algebra
and the triple
$(\g_0,\mu_0,\alpha_0)$
is both, a Hom-Lie ideal and a Hom-Lie direct summand of 
$(\g,\mu,\alpha)$; that is,
$$
(\g,\mu,\alpha) = (\g_0,\mu_0,\alpha_0)\oplus (V,\mu_V,\alpha_V).
$$
\item[(iv)] For each $1\le i\le r$, there is an element $a_i\in\g_0$,
such that, $D_i(x)=\mu(a_i,x)$ for all $x \in \g_0$.
\end{itemize}
}
\end{Prop}
\begin{proof}
{\bf (i)} Use the fact that
$V$ is contained in the center $C(\g)$ of $\g$
and the fact that $[\g,\g]\subset V^\perp=\operatorname{Im}(k)$
({\bf Lemma \ref{lema 1}.(iii)}) to conclude that 
$B([x,y],v)=0$ for all $x,y\in\g$ and all $v\in V$.
Since $k$ is injective ({\bf Lemma \ref{lema 1}.(i)}) it follows that
for any $x$ and $y$ in $\g$,
there exists a unique 
$\mu(x,y) \in \g_0$, 
such that $[x,y]=k(\mu(x,y))$.
Therefore, $h([x,y])=\mu(x,y)$ and we
obtain the following skew-symmetric bilinear map on $\g$:
\begin{equation}\label{hom lie producto}
\begin{array}{rccl}
\mu:&\g \times \g & \longrightarrow &\g\\
&(x,y)& \mapsto & h([x,y]),
\end{array}
\end{equation}
Let $\alpha:\g \to \g$ be the linear map defined by
$$
\alpha(x+v)=\alpha_0(x)+v,\quad \text{for all}\ x \in \g_0\ \text{and for all}\ v \in V.
$$
Since $V \subset C(\g)$, {\bf Lemma \ref{lema 2}.(iii)} implies that
for any $x \in \g_0$, $v \in V$ and $y,z \in \g$,
$$
\aligned
\,\mu(\alpha(x+v),\mu(y,z))&
=h([\alpha_0(x),\mu(y,z)])
\\
\,&=h([k(x),\mu(y,z)])=h\circ k([x,\mu(y,z)]_0)\\
\,&=[x,\mu(y,z)]_0=[x,h([y,z])]_0\\
\,&=h([x,[y,z]]).
\endaligned
$$
The Jacobi identity for the Lie bracket $[\,\cdot\,,\,\cdot\,]$ of $\g$,
implies that for any $x$, $y$ and $z$ in $\g$,
$$
\mu(\alpha(x),\mu(y,z))+\mu(\alpha(y),\mu(z,x))
+\mu(\alpha(z),\mu(x,y))=0.
$$
This proves that $(\g,\mu,\alpha)$ is a Hom-Lie algebra. 
Observe that $\mu(\g,\g) \subseteq \g_0$, because $\Im(h)=\g_0$.
Finally, observe that $V \subset C(\g)$ implies $\mu(V,\g)=\{0\}$.

\medskip
{\bf (ii)} 
Let $\mu_0=\mu\vert_{\g_0\times\g_0}:\g_0\times\g_0\to\g_0$
and $\alpha_0=\pi\circ\alpha\vert_{\g_0}:\g_0\to\g_0$ be as in the statement. 
Since $\alpha(x)-k(x)\in V$ for any $x\in\g_0$ and $V \subset C(\g)$, 
it follows that,
$$
\mu(\alpha(x),\mu(y,z))=\mu_0(\alpha_0(x),\mu_0(y,z)),\quad
\text{for all\ }\,x,y,z \in \g_0.
$$ 
Therefore $(\g_0,\mu_0,\alpha_0)$ is a Hom-Lie algebra. 
Now, let $B_0$ be the given invariant metric for the Lie algebra $\g_0$
and let $B$ be the invariant metric on the Lie algebra $\g$.
Since $\alpha_0$ is $B_0$-symmetric, 
we may use the invariance of $B$ 
with respect to the Lie bracket $[\,\cdot\,,\,\cdot\,]$ defined in \eqref{tc1}
and {\bf Lemma \ref{lema 1}.(ii)}, to deduce that,
$$
\aligned
B_0(\mu_0(x,y),z)
& = B_0\left(h([x,y]),z\right) = B([x,y],z)
\\
& = B(x,[y,z]) = B_0(x,h([y,z]))
= B_0\left(x,\mu_0(y,z)\right),
\endaligned
$$
thus proving {\bf (ii)}.

\medskip
{\bf (iii)} The statement follows from $\mu(\g,\g) \subset \g_0$,
$\alpha_0=\alpha\vert_{\g_0}$,
and $\alpha_V=\operatorname{Id}_V$.

\medskip
{\bf (iv)} The statement follows from
\eqref{definicion rho} and {\bf Lemma \ref{lema 2}.(iv)}.
\end{proof}
\medskip

\begin{Remark}{\rm
Since the skew-symmetric bilinear map $\mu:\g \times \g \to \g$
just defined in \eqref{hom lie producto} satisfies $\mu(\g,\g)\subset \g_0$,
we might as well use the same symbol $\mu$ to denote its
restriction $\mu_0=\mu\vert_{\g_0 \times \g_0}$ to $\g_0$.
We shall adhere ourselves to this convention
whenever it turns convenient to do so.}
\end{Remark}
\medskip

\begin{Remark}{\rm
The linear map $\alpha:\g\to\g$ that works as twist map for the 
skew-symmetric product $\mu$ in $\g$ 
may be considered as an extension of the linear map
$\alpha_0:\g_0\to\g_0$,
which is in turn determined by $k:\g_0\to\g$.
The map $\alpha$, however, is not a unique extension for $\alpha_0=\pi\circ k$. 
The following result shows that there is at least
a different extension $\alpha^\prime:\g\to\g$
that makes $(\g,\mu,\alpha^\prime)$ into a Hom-Lie algebra.
}
\end{Remark}
\medskip

\begin{Cor}\label{corolario}{\sl
Let the hypotheses be as stated in \S1 above
and let $k:\g_0\to\g$ be defined as in \eqref{maps}.
Let $\alpha^\prime:\g \to \g$, be given by,
$\alpha^\prime(x+v)=k(x)$
for all $x \in \g_0$ and for all $v \in V$.
Then, $(\g,\mu,\alpha^\prime)$ is a Hom-Lie algebra, 
and $\alpha^\prime$ satisfies,
for any $x$ and $y$ in $\g$,
\begin{equation}\label{centroide L}
\alpha^\prime([x,y])
=[\alpha^\prime(x),y],
\ \ \text{and}\ \ 
B(\alpha^\prime(x),y)=B(x,\alpha^\prime(y)),
\end{equation}
together with,
\begin{equation}\label{delta producto-1}
\alpha^\prime(\mu(x,y))=[x,y].
\end{equation}
Furthermore, for any $x$ and $y$ in $\g_0$,
\begin{equation}\label{delta producto-2}
\mu(\alpha^\prime(x),y)=[x,y]_0.
\end{equation}
The linear map $\alpha_0$ of the 
Hom-Lie algebra $(\g_0,\mu_0,\alpha_0)$ satisfies:
\begin{equation}
\label{delta producto-3}
\alpha_0(\mu_0(x,y))=\mu_0(\alpha_0(x),y)=[x,y]_0,\ \ \ \text{for all\ }x,y\in\g_0.
\end{equation}
Moreover, $(\g_0,\mu_0,\alpha_0)$
is an ideal of $(\g,\mu,\alpha^\prime)$ if and only if
the 2-cocycle $\theta:\g_0\times\g_0\to V$ 
of the extension is equal to zero.
}
\end{Cor}

\medskip
\begin{proof}
Let $\alpha:\g \to \g$ be the twist map of
{\bf Prop. \ref{Prop 3}}. Since $\mu(V,\g)=\{0\}$,
it is clear that for any $x,y,z,\in\g$,
$$
\aligned
\mu(\alpha(x),\mu(y,z)) & =\mu(\alpha_0(x),\mu(y,z))=\mu(k(x),\mu(y,z))
\\
& =\mu(\alpha^\prime(x),\mu(y,z)).
\endaligned
$$
It then follows that $(\g,\mu,\alpha^\prime)$ is a Hom-Lie algebra.

\medskip
To prove \eqref{centroide L} we claim that
$\alpha^\prime\in\Gamma_B(\g)$. This follows from
the fact $\alpha^\prime(x)=k(x)$ for all $x\in\g_0$
and that $\Ker(\alpha^\prime)=V\subset C(\g)$
according to \textbf{Lemma \ref{lema 1}.(ii)-(v)}.
On the other hand, since $\mu(\,\cdot\,,\,\cdot\,) = h([\,\cdot\,,\,\cdot\,])$,
we may use 
{\bf Lemma \ref{lema 1}}
to conclude that $[x,y]=
\alpha^\prime(\mu(x,y))$, for all $x,y \in \g$,
which proves \eqref{delta producto-1}. 
Similarly, {\bf Lemma \ref{lema 1}}, also shows that for any $x,y \in \g_0$,
$$
\mu(\alpha^\prime(x),y)=h([\alpha^\prime(x),y])=h([k(x),y]_0)=[x,y]_0,
$$
which proves \eqref{delta producto-2}. 
To prove \eqref{delta producto-3}, observe
that $\alpha_0=\pi \circ \alpha^{\prime}$. 
Then, \eqref{delta producto-1} implies that for
any $x$ and $y$ in $\g_0$, 
$\alpha_0(\mu_0(x,y))=[x,y]_0$. On the other hand,
using \eqref{delta producto-2},
we get $\mu_0(\alpha_0(x),y)=[x,y]_0$, as stated.

\medskip
Finally, from {\bf Prop. \ref{Prop 3}.(i)}, we know that 
$\mu(\g_0,\g) \subset \g$.
To test if $\g_0$ is a Hom-Lie ideal of $(\g,\mu,\alpha^\prime)$,
it only remains to check if $\g_0$ is invariant under $\alpha^\prime$.
From the definition of $\alpha^\prime$ and \eqref{k}, it follows that $\g_0$
is invariant under $\alpha^\prime$ if and only if $a_1=\cdots =a_r=0$
and from {\bf Prop. \ref{Prop 3}.(iv)} this implies that $D_1=\cdots=D_r=0$.
\end{proof}

\medskip
\begin{Remark}\label{remark 1}
{\rm
Observe that if $\g_0$ is a semisimple Lie algebra, the sequence \eqref{m} splits.
In that case $[\g,\g]=\g_0$ and $C(\g)=V$; thus, $\g_0$ and $V$ are ideals in $\g$.
However, if $\g_0$ is not semisimple, \eqref{m} does not split
(see {\bf Example 2.16} in \cite{GSV}).
For the Hom-Lie algebra described in {\bf Prop. \ref{Prop 3}}, 
$\g_0$ is an ideal and a direct sumand of $\g$, with $\mu(\g,\g) \subset \g_0$,
regardless of whether $\g_0$ is semisimple or not. 
If we consider the twist map $\alpha^\prime$, however, 
{\bf Cor. \ref{corolario}} shows that $\g_0$ is no longer an ideal
of the Hom-Lie algebra $\g$, unless the 2-cocycle $\theta$ be zero.
Nevertheless, the conditions given in \eqref{centroide L}, \eqref{delta producto-1} and \eqref{delta producto-2}
will allow us to state a deeper result (see {\bf Thm. \ref{teorema}} below) concerning
the Hom-Lie algebra structure given in {\bf Cor. \ref{corolario}}.
}
\end{Remark}

The existence of the invariant metrics 
$B_0:\g_0 \times \g_0 \to \F$ and $B:\g \times \g \to \F$, 
was strongly used in \eqref{maps}  to define
the linear maps $h:\g \to \g_0$ and $k:\g_0 \to \g$
whose properties have been thoroughly used to
define the Hom-Lie algebra structures exhibited
on the underlying spaces $\g_0$ and $\g$.
It is natural to ask whether or not one can still produce 
Hom-Lie algebra structures on the Lie algebras
$\g_0$ and $\g$ related to each other {\it only\/} via \eqref{m};
that is, only through a given $2$-cocycle $\theta$
{\it with no assumption whatsoever\/} on the existence of
any invariant metrics $B_0$ and $B$ on $\g_0$ and $\g$, respectively.
We now settle this question precisely in the following:
\medskip

\begin{Prop}\label{sin-metricas}{\sl
Let $\g_0$ be a Lie algebra, and let $V$ be a finite dimensional vector space.
Assume $\g=\g_0\oplus V$ is a central extension of $\g_0$ by $V$
associated to a given $2$-cocycle $\theta\in Z^2(\g_0,V)$.
Let $[\,\cdot\,,\,\cdot\,]_0$ and $[\,\cdot\,,\,\cdot\,]$ 
be the Lie brackets on $\g_0$ and $\g$, respectively.
Let $h:\g \to \g_0$ and $k:\g_0 \to \g$ be linear maps satisfying:

\begin{itemize}
\item[(i)] 
$k([x,y]_0)=[k(x),y]$, for all $x,y \in \g_0$

\item[(ii)] 
$k(h([x,y]))=[x,y]$, for all $x,y \in \g$.

\end{itemize}
Then, there is a linear map $\alpha:\g\to\g$ which,
together with the skew-symmetric bilinear map 
$\mu:\g\times\g\to\g$ given by $\mu(x,y)=h([x,y])$,
makes the triple $(\g,\mu,\alpha)$ into a Hom-Lie algebra.
Moreover, by
letting $\mu_0=\mu\vert_{\g_0\times\g_0}$ and
$\alpha_0 = \pi\circ k$, where
$\pi:\g\to\g_0$ is the projection map in \eqref{m},
the triple
$(\g_0,\mu_0, \alpha_0)$
becomes a Hom-Lie subalgebra.
Actually, $\alpha:\g\to\g$ can be written in the form
$\alpha(x+v)=\alpha_0(x)+\sigma(x+v)$, for some linear map $\sigma:\g\to V$.
If $\sigma\vert_{\g_0}=0$, then 
$(\g_0,\mu_0,\alpha_0)$ is actually 
a Hom-Lie ideal and a direct summand of $(\g,\mu,\alpha)$
with  $\mu\vert_{V\times V}=0$ and,
$$
(\g,\mu,\alpha)=(\g_0,\mu_0,\alpha_0)
\oplus
(V,\mu\vert_{V\times V},\sigma\vert_V).
$$
}
\end{Prop}

\medskip
\begin{proof}
Set $\mu(x,y)=h([x,y])$ for any pair $x,y\in\g$.
Since $V\subset C(\g)$,
it follows that $h([v,x])=0$ for any $x \in \g$ and $v\in V$.
Let $\sigma:\g\to V$ be a linear map and
write $\alpha:\g\to\g$ as in the statement; namely,
$\alpha(x+v)=\alpha_0(x)+\sigma(x+v)$ for any $x\in\g_0$ and $v\in V$.
If $x \in \g_0$, $v \in V$ and $y,z \in \g$, it follows that,
$$
\aligned
\mu(\alpha(x+v),\mu(y,z))
& =\mu(\alpha_0(x)+\sigma(x+v),h([y,z])) \\
& =\mu(\alpha_0(x),h([y,z])) =\mu(k(x),h([y,z])) \\
& =h([k(x),h([y,z])])=h([x,k(h([y,z]))])\\
&=h([x,[y,z]])=h([x+v,[y,z]]).
\endaligned
$$
Thus, for any triple $x,y,z\in\g$ we have,
$$
\sum_{\circlearrowright}\mu(\alpha(x),\mu(y,z))=\sum_{\circlearrowright}h([x,[y,z]])=0.
$$
Therefore, the triple $(\g,\mu,\alpha)$ is a Hom-Lie algebra for the particular $\mu$
given by $\mu(x,y)=h([x,y])$ and $\alpha=\pi\circ k\circ\pi +\sigma=\alpha_0\circ\pi+\sigma$.
It is now easy to see that the triple $(\g_0,\mu_0, \alpha_0)$
is a Hom-Lie subalgebra with the property that $\mu(\g,\g)\subset\g_0$.
To see whether or not this triple
is an ideal in $(\g,\mu,\alpha)$ one needs to check whether or not
$\g_0$ is an invariant subspace for $\alpha:\g\to\g$. From the definition of $\alpha$,
it follows that
$\g_0$ is $\alpha$-invariant if and only if $\sigma\vert_{\g_0}=0$,
thus proving the last statement.
\end{proof}
\medskip

\begin{Remark}{\rm
Observe that {\it any\/} twist map 
$\alpha:\g \to \g$ for the Hom-Lie multiplication map $\mu$ satisfies
$\alpha_0=\pi \circ \alpha\vert_{\g_0}=\pi \circ k\vert_{\g_0}$. 
In particular, two twist maps $\alpha$ and $\alpha^\prime$
for the same $\mu$
can differ only by a linear map $\sigma:\g \to V$. Therefore,
the Hom-Lie algebra structure on $\g$ of {\bf Prop \ref{Prop 3}}
corresponds to the case $\sigma(x)=0$ for all $x \in \g_0$.
On the other hand, the Hom-Lie algebra structure on $\g$ of {\bf Cor. \ref{corolario}}
corresponds to the case $\sigma(x)=\pi_V\circ k(x)$, for all $x \in \g_0$.}
\end{Remark}

\medskip
We have pointed out in {\bf Remark \ref{remark 1}} that 
$\mu(\g,\g)\subset\g_0$ and $V\subset C(\g)$.
If we add the hypothesis that the $2$-cocycle
of the extension $\g$ of $\g_0$ by $V$ is not a coboundary,
then we actually have equalities: $\mu(\g,\g)=\g_0$ and $V=C(\g)$.
Now recall from \cite{GSV} that if $\g$ admits an invariant metric
and $\dim_{\F}(V) \leq 2$, then \eqref{m} splits and therefore
$\theta$ is a coboundary, which implies 
that $\g_0$ is actually a Lie algebra ideal of $\g$.
Thus, if we assume that $\theta$ is not a coboundary,
then $\dim_{\F}(V)>2$ under the assumption that
$\g$ admits an invariant metric.
\medskip

\begin{Prop}\label{perfect hom-Lie algebra}{\sl
Let the hypotheses be as in {\bf Prop \ref{Prop 3}}. 
Suppose that at least one of the $B$-skew-symmetric derivations 
$D_1$, $\ldots$, $D_r$, is not inner.
Let $C_{\mu}(\g)=\{x \in \g\mid h([x,\g])=\{0\}=\mu(x,\g)\}$. 
Then,
\begin{itemize}

\item[(i)] $C(\g)=C_{\mu}(\g)=V$;

\smallskip
\item[(ii)] $\g_0=h([\g,\g])=\mu(\g,\g)=\mu(\g_0,\g_0)$;

\smallskip
\item[(iii)]
$C(\g_0)=\Ker(\alpha_0)$ and $[\g_0,\g_0]_0=\Im(\alpha_0)$;

\smallskip
\item[(iv)] If $x \in \g_0$ satisfies $\theta(x,y)=0$ for all $y \in \g_0$, then $x=0$.

\end{itemize}
}
\end{Prop}
\begin{proof}
{\bf (i)}  To say that one of the $B_0$-skew-symmetric derivations
$D_i$, $1\le i\le r$, is not inner, is equivalent to say that the 
$2$-cocycle $\theta$ of the extension is not a coboundary. 
We shall refer the reader to \cite{GSV}
(specifically, {\bf Lemma 2.2} and {\bf Prop. 2.5.(ii)} therein)
where it is proved that if at least one $D_i$ is not inner, then
$\displaystyle{\cap_{i=1}^r\Ker(D_i)}=\{0\}$. Since
$C(\g)=C(\g_0) \cap (\cap_{i=1}^r\Ker(D_i)) \oplus V$,
it follows that $C(\g)=V$ when $\theta$ is not a coboundary.
It remains to prove that $C_{\mu}(\g)=V$.
This is done in the first part of {\bf (ii) below.}

\medskip
{\bf (ii)} Let $x \in \g_0$ be such that $h([x,y])=\mu(x,y)=0$, for all $y \in \g_0$.
Then,  $\mu(x,\g)\in\operatorname{Ker}(h)\cap V^\perp = \g_0^\perp\cap V^\perp
 = \{0\}$. By {\bf Lemma \ref{lema 1}.(iii)}, $x \in C(\g)=V$ implies $x=0$.
Whence, $C_{\mu}(\g)=V$. Finally, it was proved in {\bf Prop. \ref{Prop 3}}
that $B_0$ is invariant under $\mu_0$. 
Thus, $C_{\mu}(\g_0)^{\perp}=\mu(\g_0,\g_0)$ and
$\g_0=\mu(\g_0,\g_0)$.

\medskip
{\bf (iii)}
Recall that $\alpha_0=\pi\circ k$.
We deduce
from {\bf Lemma \ref{lema 2}.(vi)-v(ii))}
that $\alpha_0\left(h([\g_0,\g_0])\right) 
=[\g_0,\g_0]_0$.
Whence, 
$\alpha_0(\g_0)=[\g_0,\g_0]_0$,
which proves {\bf (iii)}. 
Using the fact that
$\alpha_0\in\Gamma_{B_0}(\g_0)$
together with {\bf Lemma \ref{lema 2}.(ii)}, we obtain, 
$\Ker(\alpha_0)=C(\g_0)$.

\medskip
{\bf (iv)} The result follows from the fact that 
$\theta(x,y)=0$ for all $y \in \g$, if and only if $x\in \cap_{i=1}^r\Ker(D_i)=\{0\}$.
\end{proof}
\medskip

\medskip
\begin{Remark}\label{remark 2}{\rm
As far as the Hom-Lie algebra structure on $\g$ is concerned,
the statements in {\bf Prop. \ref{perfect hom-Lie algebra}}
are independent of the twist map. Thus they hold true either with
$\alpha$ (defined in {\bf Prop. \ref{Prop 3}})
or with $\alpha^\prime$ (defined in {\bf Cor. \ref{corolario}}).
}
\end{Remark}

\medskip
\section{Killing-like forms for Hom-Lie algebras}

\medskip
We shall focus ourselves on property
\eqref{delta producto-3} of {\bf Cor. \ref{corolario}}.
The aim of this section is to prove in general that, in any 
Hom-Lie algebra $(\g,\mu,\alpha)$ with an underlying
Lie bracket $[\,\cdot\,,\,\cdot\,]$ on $\g$, if 
\begin{equation}\label{delta producto-3 otra vez}
\alpha(\mu(x,y))=\mu(\alpha(x),y)=[x,y],
\end{equation}
holds true for any $x$ and $y$ in $\g$, 
one may produce 
a symmetric bilinear form $K:\g\times\g\to\Bbb F$ satisfying,
$$
K\left(\mu(x,y),z\right) = K\left(x,\mu(y,z)\right).
$$
The next result shows that $K$ is actually given by a trace formula
involving only the maps $\ad_\mu(x)=\mu(x,\,\cdot\,):\g\to\g$
and $\alpha:\g\to\g$,
exhibiting $K$ as a natural generalization of the Cartan-Killing
form of a Lie algebra $\g$.

\medskip
\begin{Theorem}\label{Teorema Cartan-Killing}
{\sl
Let $(\g,[\,\cdot\,,\,\cdot\,])$ be a Lie algebra for which there exists
a skew-symmetric bilinear map $\mu:\g \times \g \to \g$ and a linear map
$\alpha:\g \to \g$, turning $(\g,\mu,\alpha)$ into a Hom-Lie algebra 
satisfying \eqref{delta producto-3 otra vez}.
For each $x\in\g$ write $\ad_{\mu}(x)=\mu(x,\,\cdot\,):\g\to\g$.
Then, the assignment,
$$
(x,y)\ \mapsto\ K(x,y)=\Tr\left(\ad_{\mu}(x)\circ\ad_{\mu}(y)\circ\alpha\right),
\quad\text{for all\ }x,y\in\g,
$$
defines a symmetric, bilinear form on $\g$ which further satisfies:
$$
\alignedat 3
K(\alpha(x),y) & =K(x,\alpha(y))=\kappa(x,y),\quad &\quad &\text{for all\ }x,y\in\g;
\\
K\left(\mu(x,y),z\right) & =K\left(x,\mu(y,z)\right),\quad &\quad &\text{for all\ }x,y,z\in\g;
\\
K\left([x,y],z\right) & =K\left(x,[y,z]\right),\quad &\quad &\text{for all\ }x,y,z\in\g.
\endalignedat
$$
Moreover, 
\begin{itemize}
\item[(i)] $K$ is non-degenerate if and only if
the Cartan-Killing form $\kappa$ of $\g$ is.
\item[(ii)] If $K\left(x,[y,z]\right)=0$ for all $x,y,z \in \g$, then $\g$ is a solvable Lie algebra. 
\item[(iii)] $K=0$, if and only if $\g$ is a nilpotent Lie algebra. 
\end{itemize}
}
\end{Theorem}

\medskip
\begin{proof}
Combining \eqref{delta producto-3 otra vez}
with the Hom-Lie Jacobi identity
satisfied by $(\g,\mu,\alpha)$, we get,
\begin{equation}\label{traza2}
\displaystyle{\sum_{\circlearrowleft\{x,y,z\}}}
[x,\mu(y,z)]=0,\quad \text{for all\ }\,x,y,z \in \g.
\end{equation}
Thus, for any $x,y,z \in\g$, we deduce that,
\begin{eqnarray}
\label{traza3}\,[x,\mu(y,z)]&=&[\mu(x,y),z]+[y,\mu(x,z)],\\
\label{traza4}\,\mu\left(x,[y,z]\right)&=&\mu\left([x,y],z\right) + \mu\left(x,[y,z]\right).
\end{eqnarray}
It follows from
\eqref{delta producto-3 otra vez} 
that for any $x,y\in\g$,
\begin{equation}\label{traza6}
K(x,y)=\Tr(\ad(x) \circ \ad_{\mu}(y))=\Tr(\ad_{\mu}(x) \circ \ad(y)).
\end{equation}
Similarly, using
\eqref{delta producto-3 otra vez} 
we get,
\begin{equation}\label{relacion K y cartan-killing}
\aligned
K(\alpha(x),y)&=\operatorname{Tr}(\ad_{\mu}(\alpha(x)) \circ \ad(y))\\
\,&=\operatorname{Tr}(\ad(x) \circ \ad(y))=\kappa(x,y),\quad\text{for all\ }\,x,y \in \g.
\endaligned
\end{equation}
It is also true that for any $x,y,z \in \g$,
\begin{equation}
\label{inv-K_0}
K\left(\mu(x,y),z\right)=K\left(x,\mu(y,z)\right).
\end{equation}
To prove this claim, let $w\in\g$ be arbitrary. Then \eqref{traza4} implies that,
$$
\ad_{\mu}([x,y]) \circ \ad_{\mu}(z)(w)
=\mu\left(x,[y,\mu(z,w)]\right)-\mu\left(y,[x,\mu(z,w)]\right).
$$
Write $s_1(w)=\mu\left(x,[y,\mu(z,w)]\right)$ 
and $s_2(w)=\mu\left(y,[x,\mu(z,w)]\right)$ for all $w \in \g$.
Using \eqref{delta producto-3 otra vez} we may write, 
$s_2=\ad_{\mu}(y) \circ \ad_{\mu}(x) \circ \ad_{\mu}(z) \circ \alpha$.
On the other hand, by \eqref{traza4} and 
\eqref{delta producto-3 otra vez},
we have,
$$
\aligned
(\,\ad_{\mu}(x) \,\circ\, &
\ad_{\mu}{\left(\mu(y,z)\right)} \circ \alpha\,)\,(w)
=\ad_{\mu}(x) \circ \ad_{\mu}([y,z])(w)
\\
&=\mu\left(x,{\mu\left(y,[z,w]\right)}\right)-\mu\left(x,{\mu\left(z,[y,w]\right)}\right).
\endaligned
$$
Now write $t_1(w)=\mu\left(x,\mu(y,[z,w])\right)$
and $t_2(w)=\mu\left(x,\mu(z,[y,w])\right)$ for all $w \in \g$.
Again, using \eqref{delta producto-3 otra vez}
we may write, 
$t_2=\ad_{\mu}(x) \circ \ad_{\mu}(z) \circ \ad_{\mu}(y) \circ \alpha$. 
From \eqref{delta producto-3 otra vez}
we conclude that, $s_1(w)=t_1(w)$ for all $w \in \g$, 
so $t_1=s_1$ and $\Tr(s_1)=\Tr(t_1)$. 
Since $\alpha$ commutes with $\ad_{\mu}$, we obtain,
$$
\aligned
\Tr(s_2)&=\Tr(\ad_{\mu}(y) \circ \ad_{\mu}(x) \circ\ad_{\mu}(z) \circ \alpha)
\\
&=\Tr(\ad_{\mu}(x) \circ \ad_{\mu}(z) \circ \alpha \circ \ad_{\mu}(y))
\\
&=\Tr(\ad_{\mu}(x) \circ\ad_{\mu}(z) \circ\ad_{\mu}(y) \circ \alpha)
=\Tr(t_2).
\endaligned
$$
Therefore, $K\left(\mu(x,y),z\right)=K\left(x,\mu(y,z)\right)$ as claimed. 
Now we shall prove that $K([x,y],z)=K(x,[y,z])$, for all $x,y,z \in \g$. 
Observe that \eqref{delta producto-3 otra vez}
implies that $\alpha$ is $\ad_{\mu}$-equivariant.
Then, $\alpha$ is $K$-symmetric; 
that is, $K(\alpha(x),y)=K(x,\alpha(y))$, for all $x,y \in \g$. 
Now, from \eqref{delta producto-3 otra vez}
we deduce that,
$$
\aligned
K([x,y],z)&=K(\alpha(\mu(x,y)),z)=K(\mu(x,y),\alpha(z))\\
\,&=K(x,\mu(y,\alpha(z)))=K(x,[y,z]),\quad\text{for all\ }\,x,y,z \in \g.
\endaligned
$$
{\bf(i)} We shall now prove that $K$ is non-degenerate if and only if
the Cartan-Killing form $\kappa$ is non-degenerate.
Suppose $\kappa$ is non-degenerate.
It then follows from \eqref{relacion K y cartan-killing} that $\alpha$ is invertible. 
Therefore, $K(\,\cdot\,,\,\cdot\,)=\kappa\left(\alpha^{-1}(\,\cdot)\,,\,\cdot\,\right)$ is non-degenerate.

\medskip
Conversely: suppose $K$ is non-degenerate. 
We claim that $\alpha$ is invertible. 
Let $x \in \Ker(\alpha)$ and take any $y \in \g$.
Since $\alpha$ is $\ad_{\mu}$-equivariant, 
$K(x,y)=\operatorname{Tr}\left(\ad_{\mu}(\alpha(x))\circ \ad_{\mu}(y)\right)=0$, for any $y\in\g$.
Since $K$ is non-degenerate, $x=0$, thus implying $\Ker(\alpha)=\{0\}$,
so that $\alpha$ is invertible. 
Finally, using \eqref{relacion K y cartan-killing}, 
we conclude that $\kappa$ is non-degenerate.

\medskip
{\bf (ii)} For what remains to be proved, use will be made of the following basic fact 
(see Humphreys, \cite{Hum}; {\bf Lemma 4.3}): Let $V$ be a finite dimensional
vector space over $\Bbb F$ and let $A$ and $B$ be two subspaces with $A\subset B$.
Set $M=\{X \in\frak{gl}(V)\,\mid\,[X,B]_{\frak{gl}(V)} \subset A\}$. 
{\sl If $X \in M$ satisfies $Tr(XY)=0$ for all $Y \in M$, then $X$ is nilpotent\/.}
We may use this result in the following context:
$$
V=\g,\quad A=\ad_{\mu}([\g,\g])\quad\text{and}\quad B=\ad_{\mu}(\g).
$$
Clearly, $A\subset B$. Also, by
\eqref{delta producto-3 otra vez},
it is clear that $\ad(\g) \subset M$ where,
$$
M=\{X \in \frak{gl}(\g)\,\mid\,[X,\ad_{\mu}(\g)]_{\frak{gl}(\g)} \subset 
\ad_{\mu}([\g,\g]\}.
$$
We now proceed with the proof of the stated items.

\medskip
Let $x,y \in\g$ and $Z \in M$. Using \eqref{delta producto-3 otra vez} 
we get,
$$
\aligned
\Tr(\ad_{\mu}([x,y]) \circ Z)
& =\Tr([\ad(x),\ad_{\mu}(y)]_{\frak{gl}(\g)} \circ Z)
\\
& =\Tr(\ad(x) \circ [\ad_{\mu}(y),Z]_{\frak{gl}(\g)})=0.
\endaligned
$$
Therefore, $\ad_{\mu}([x,y])=\ad(\mu(x,y))$ is nilpotent for all $x,y \in \g$.
This means that, $\mu(\g,\g) \subset \Nil(\g)$
---{\it the maximal nilpotent ideal of the Lie algebra $\g$\/.} 
Since $\alpha([x,y])=\alpha^2(\mu(x,y))=\mu(\alpha^2(x),y)
=[\alpha(x),y]$, we conclude that $\alpha$ is $\ad$-equivariant.
This implies that $\Nil(\g)$ is invariant under $\alpha$, 
so by \eqref{delta producto-3 otra vez}
we get, 
$[\g,\g] \subset \alpha_0\left(\mu(\g,\g)\right) \subset \Nil(\g)$ and
therefore, $\g$ is a solvable Lie algebra (see \cite{Hum}, {\bf Prop. 3.1}).

\medskip
{\bf (iii)} If $K=0$, then by \eqref{relacion K y cartan-killing}, 
we have that $\kappa=0$ and therefore $\g$ is a nilpotent Lie algebra.
We shall now prove that 
if $\g$ is a nilpotent Lie algebra, then $K=0$.
First observe that
\eqref{delta producto-3 otra vez}
states that $\alpha$ is 
$\ad$-equivariant and $\Ker(\alpha) \subset C(\g)$.
We claim that $\Ker(\alpha)$ is a Hom-Lie ideal of $\g$ and for this,
we only need to prove that $\mu(\Ker(\alpha),\g) \subset \Ker(\alpha)$.
Let $x \in \Ker(\alpha)$ and $y \in \g$. 
Since $\Ker(\alpha) \subset C(\g)$, using
\eqref{delta producto-3 otra vez} 
we get,
$\alpha(\mu(x,y))=[x,y]=0$, thus proving that $\mu(x,y) \in \Ker(\alpha)$.

\medskip
Now assume that $\g$ is a nilpotent Lie algebra. 
Then, $\kappa=0$ and by \eqref{relacion K y cartan-killing} it follows that  
$K\left(\alpha(\,\cdot\,),\,\cdot\,\right)=0$. 
If $\alpha$ is invertible, we conclude that $K=0$.
On the other hand, assume $\Ker(\alpha) \ne \{0\}$. 
We shall now proceed by induction on $\dim_{\F}\g$. 
Let ${\frak p}$ be a complementary vector subspace of $\Ker(\alpha)$ in $\g$,
so that, $\g={\frak p} \oplus \Ker(\alpha)$. 
Let $\mu_{\frak p}(x,y) \in {\frak p}$ be the component of $\mu(x,y)$ along ${\frak p}$.
Upon restriction of $x$ and $y$ to elements from ${\frak p}$ only,
one obtains a skew-symmetric bilinear map, $\mu_{\frak p}:{\frak p} \times {\frak p} \to {\frak p}$. 
Similarly, let $\alpha_{\frak p}(x)$ be the component along ${\frak p}$
of $\alpha(x)$. By restriction to ${\frak p}$ one obtains a linear map,
$\alpha_{\frak p}:{\frak p}\to{\frak p}$, and the triple
$({\frak p},\mu_{\frak p},\alpha_{\frak p})$ becomes a Hom-Lie algebra.

\medskip
In the same way, let $[x,y]_{\frak p}$ be the ${\frak p}$-component of the Lie bracket $[x,y]$,
for any $x,y \in {\frak p}$. Since $\Ker(\alpha)$ is a Hom-Lie ideal of $(\g,\mu,\alpha)$,
\begin{equation}\label{condicion para p}
\alpha_{\frak p}\left(\mu_{\frak p}(x,y)\right)=
\mu_{\frak p}\left(\alpha_{\frak p}(x),y\right)=
[x,y]_{\frak p},\quad \text{for all\ }x,y \in {\frak p}.
\end{equation}
Observe that $({\frak p},[\cdot,\cdot]_{\frak p})$ is a nilpotent Lie algebra, 
because ${\frak p}\simeq\g/\Ker(\alpha)$ and $\Ker(\alpha) \subset C(\g)$.
It clearly follows from \eqref{condicion para p}
that $\Ker(\alpha_{\mathfrak{p}}) \subset C(\mathfrak{p})$.

\medskip
Finally, let $x,y,z \in {\frak p}$. Since $\Ker(\alpha)$ is a Hom-Lie ideal of $\g$, we have,
$$
\aligned
\ad_{\mu}(x) \circ \ad_{\mu}(y) \circ \alpha(z)
&=\mu\left(x,\mu(y,\alpha(z))\right)
\\
& =\mu(x,[y,z]) =\mu\left(x,[y,z]_{\frak p}\right) +\Ker(\alpha)
\\
& = \mu_{\frak p}\left(x,[y,z]_{\frak p}\right) + \Ker(\alpha)
\\
&=\mu_{\frak p}\left(x,{\mu_{\frak p}\left(y,\alpha_{\frak p}(z)\right)}\right) + \Ker(\alpha_{\frak p})
\\
&=\ad_{\mu_{\frak p}}(x) \circ \ad_{\mu_{\frak p}}(y) \circ \alpha_{\frak p}(z) + \Ker(\alpha_{\frak p}).
\endaligned
$$
Let $K_{\frak p}:{\frak p} \times {\frak p} \to \Bbb F$
be the corresponding symmetric bilinear form on ${\frak p}$ given by:
$$
K_{\frak p}(x,y)=\operatorname{Tr}\left(\ad_{\mu_{\frak p}}(x) \circ 
\ad_{\mu_{\frak p}}(y) \circ \alpha_{\frak p}\right),\quad \text{for all\ }x,y \in {\frak p}.
$$
Then, for any $x$ and $y$ in ${\frak p}$,
$$
\aligned
K(x,y)
&=\operatorname{Tr}(\ad_{\mu}(x) \circ \ad_{\mu}(y) \circ \alpha)
\\
& = \operatorname{Tr}(\ad_{\mu_{\frak p}}(x) \circ \ad_{\mu_{\frak p}}(y) \circ \alpha_{\frak p})
 =K_{\frak p}(x,y).
\endaligned
$$
This proves that ${\frak p}$ satisfies the same conditions that $\g$ does.
By induction on $\dim_{\F}\g$, we conclude that $K_{\frak p}(x,y)=0$, for all $x,y \in {\frak p}$.
Since $\g={\frak p} \oplus \Ker(\alpha)$, it follows that $K(x,y)=0$ for all $x,y \in \g$.
\end{proof}

\medskip
One would like to have a result relating the
algebraic structure of a given Hom-Lie algebra $(\g,\mu,\alpha)$
in terms of some property of its associated bilinear form
$K$. The following result goes in this direction.
It states that if $(\g,\mu,\alpha)$ is nilpotent,
then the Lie algebra $(\g,[\,\cdot\,,\,\cdot\,])$ is nilpotent and therefore,
the Cartan-Killing-like bilinear form $K$ of $(\g,\mu,\alpha)$ is zero.
The converse, however, is not true.
If the twisted $K$ of $(\g,\mu,\alpha)$ is zero one cannot conclude that
the Hom-Lie algebra $(\g,\mu,\alpha)$ itself is nilpotent
(see \S5 below).

\medskip
\begin{Prop}\label{GilStatement}{\sl 
Let the hypotheses be as in {\bf Thm. \ref{Teorema Cartan-Killing}}. 
If the Hom-Lie algebra $(\g,\mu,\alpha)$ is nilpotent,
then the Lie algebra $(\g,[\,\cdot\,,\,\cdot\,])$ is nilpotent and therefore,
the Cartan-Killing-like bilinear form $K$ of $(\g,\mu,\alpha)$ is zero.
}
\end{Prop}

\medskip
\begin{proof}
Let $(\g,\mu,\alpha)$ be a Hom-Lie algebra satisfying
\eqref{delta producto-3 otra vez}. 
Define recursively, $\g_{\mu}^{n+1}:=\mu(\g, \g_{\mu}^n)$, 
with $\g_{\mu}^2 :=\mu(\g, \g)$.
The Hom-Lie algebra $(\g,\mu,\alpha)$ is nilpotent if
there is some $m\in\Bbb N$ such that,
$\g_{\mu}^m = 0$. It is easy to see that,
$$
\alpha(\g_{\mu}^{n+1}) = \alpha\left(\mu(\g, \g_{\mu}^n)\right)= \mu\left(\g, \alpha(\g_{\mu}^n)\right).
$$
Therefore,
$$
\alpha(\g_{\mu}^{n+1}) = \mu\left(\g, \ldots, \mu{\left(\g, \alpha(\g)\right)}, \ldots \right).
$$
On the other hand, \eqref{delta producto-3 otra vez}.
states that
$\g^2 = \mu\left(\g, \alpha(\g)\right) = \alpha(\g_{\mu}^2)$. Thus,
$$
\aligned
\g^3 = [\g, \g^2] & = \alpha\left( \mu(\g, \g^2)\right)  \\
& = \alpha( \mu\left(\g, \mu{\left(\g, \alpha(\g)\right)}\right) \\
& = \mu (\g, \alpha \left(  \mu{\left(\g, \alpha(\g)\right)} \right) \\
& = \mu \left(\g, \mu {\left(\g, \alpha^2(\g) \right)} \right),
\endaligned
$$
where $\alpha^2=\alpha\circ\alpha$. One concludes that,
$\g^{n+1} = [\g, \g^n]$ is related to the Hom-Lie product $\mu$
and the twist map $\alpha$ as follows: 
$$
\g^{n+1} = \mu\left(\g, \ldots, \mu{\left(\g, \alpha^n(\g) \right)}, \ldots \right).
$$
Observe, however, that $\alpha(\g) \subset \g$ 
implies that $\alpha^2(\g) \subset \alpha(\g)$.
Thus, if $(\g, \mu, \alpha)$ is nilpotent, there is a minimal $m\in\Bbb N$
such that $\g_{\mu}^{m+1}=0$. Whence, $\alpha(\g_{\mu}^{m+1})=0$.
This implies that,
$$
\alpha(\g_{\mu}^{m+1}) = \mu(\g, \ldots, \mu(\g, \alpha(\g)) \ldots ) = 0.
$$
Using the fact that $\alpha^m(\g) \subset \alpha(\g)$ one concludes that,
$$
\g^{m+1} = \mu(\g, \ldots, \mu(\g, \alpha^m(\g) ), \ldots ) = 0.
$$
In particular, the Cartan-Killing form $\kappa$ of the Lie algebra
$(\g,[\,\cdot\,,\,\cdot\,])$ is zero. It then follows from {\bf (i)} in the proof of
 {\bf Thm. \ref{Teorema Cartan-Killing}} that $K=0$.
\end{proof}

\medskip
\section{Hom Lie algebras on quadratic central extensions with isotropic kernel}

\medskip
It was proved in \cite{GSV} that the study of any central extension
$V\hookrightarrow\g\twoheadrightarrow\g_0$
of a quadratic Lie algebra $\g_0$, with $\g$ carrying an invariant
metric itself, can be reduced
to the study of the following two extreme cases: either,
{\bf (1)} the kernel $V$ is isotropic or, 
{\bf (2)} it is non-degenerate (see \cite{GSV}, {\bf Prop 2.6}).
When $V$ is non-degenerate, the Lie algebra structure of 
$\g$ is that of the orthogonal direct sum of $\g_0$ and $V$ (see {\bf Prop. 2.5} in \cite{GSV}). 
The interesting case is when $V$ is isotropic.
In that case, the structure of $\g$ is obtained through the construction known as
{\bf double central extension} (see {\bf Thm. 3.3} in \cite{GSV}).

\medskip
We shall therefore assume that
$\g_0$ and $\g$ are quadratic Lie algebras with invariant metrics
$B_0$ and and $B$, respectively.
Observe that {\bf Prop. \ref{Prop 3}} states that $B_0$
is $\mu_0$-invariant. The natural question is whether or not $B$ is
$\mu$-invariant. Unfortunately, if it is, $\theta$ must be a coboundary
(see {\bf Lemma \ref{invariancia hom-lie}} below).
Nevertheless, in {\bf Lemma \ref{lema 31}} we shall use the Hom-Lie algebra
$(\g_0,\mu_0,\alpha_0)$ of {\bf Prop. \ref{Prop 3}}
to provide conditions that make $V$ isotropic for $B$
without implying (nor having to assume) that $\theta$ is a coboundary.
As a direct consequence, and taking the results of \cite{GSV} into account, it follows
that if $V$ is isotropic, there is an
invariant metric $\tilde{B}$ in $\g$
making isotropic a complement of $\operatorname{Im}(\alpha_0)$ in $\g_0$
(see {\bf Cor. \ref{a isotropico}} bellow).
\medskip

\begin{Lemma}\label{invariancia hom-lie}{\sl
Let $(\g_0,[\,\cdot\,,\,\cdot\,]_0)$ be a Lie algebra with
invariant metric $B_0:\g_0\times\g_0\to\Bbb F$.
Let $B:\g\times\g\to\Bbb F$ be an invariant metric for 
the central extension
$\g=\g_0\oplus V$
associated to the cocycle $\theta$.
Let $\mu$ be 
defined by $\mu(x,y)=h([x,y])$, for all $x,y \in \g$, 
with $h:\g \to \g_0$ as in \eqref{maps} of \textbf{Lemma \ref{lema 1}}. 
Then, $B\left(\mu(x,y),z\right)=B\left(x,\mu(y,z)\right)$ for all $x,y,z \in \g$
if and only if $\theta=0$.}
\end{Lemma}

\medskip
\begin{proof}
Let $D_1,\ldots, D_r\in \operatorname{Der}(\g)\cap {\frak o}(B_0)$
($r=\dim_{\F}V$) 
be the skew-symmetric derivations associated to the cocycle
$\theta$ as in  {\bf Prop. \ref{proposicion 1}}.
Using the definition of $h:\g \to \g_0$ (see \textbf{Lemma \ref{lema 1}})
together with the invariance and the non-degeneracy of both $B$ and $B_0$,
one proves that,
\begin{equation}\label{condicion para h}
h([x,y])=[h(x),y]_0+\sum_{i=1}^rB_0(x,v_i)D_i(y),
\end{equation}
which holds for all $x \in \g$ and $y \in \g_0$
(also see \textbf{Lemma 2.2} in \cite{GSV} for an alternative proof). 

\medskip
Suppose that $D_1=\cdots =D_r=0$. 
Then $\g_0$ is an ideal of $\g$, which is equivalent to $\theta=0$. 
Thus, the Lie bracket in $\g_0$ is the restriction 
of the Lie bracket on $\g$. 
By \eqref{condicion para h} it follows that 
$h([x,y])=h([x,y]_0)=[h(x),y]_0=[x,h(y)]_0$, for all $x,y \in \g$. 
Since $B(h(x),y)=B(x,h(y))$, for all $x,y \in \g$ 
(see \textbf{Lemma \ref{lema 1}}), we conclude that,
$$
\aligned
B(\mu(x,y),z)&=B([x,y],h(z))=B(x,[y,h(z)])
\\
&=B(x,[y,h(z)]_0)=B(x,h([y,z]))=B(x,\mu(y,z)).
\endaligned
$$
Conversely, suppose $B(\mu(x,y),z)=B(x,\mu(y,z))$, for all $x,y,z \in \g$.
In particular, take $x=a_i+w_i \in \Ker(h)=\g^{\perp}$ ($1\le i\le r$)
 (see \textbf{Lemma \ref{lema 2}}). 
Since $\mu(y,z) \in \g_0$, for all $y,z \in \g$, we have,
$B\left(\mu(a_i+w_i,y),z\right)=B\left(a_i+w_i,\mu(y,z)\right)=0$. 
But $D_i(y)=\mu(a_i+w_i,y)$, for all $y \in \g_0$. Then $B(D_i(y),z)=0$, 
for all $y \in \g_0$ and for all $z \in \g$.
Since $B$ is non-degenerate, it follows that $D_i=0$, for all $1 \leq i \leq r$.
\end{proof}

\medskip
The following statements 
are straightforward consequences of {\bf Lemma \ref{lema 1}},
{\bf Lemma \ref{lema 2}} and {\bf Prop. \ref{Prop 3}};
they are also proved in \cite{GSV}
(see \textbf{Lemma 2.7} and \textbf{Lemma 2.9} therein).
We need them all for the proof of
\textbf{Thm. \ref{teorema}} below.

\medskip
\begin{Lemma}\label{lema 31}{\sl
Let the hypotheses be as in {\bf Prop. \ref{Prop 3}} above.
Let $V$ be an isotropic ideal of $\g$, and let
$\g_0=\Im(\alpha_0)+{\frak a}$, 
with  
$\alpha_0=\pi\circ k:\g_0\to\g_0$ as in {\bf Prop. \ref{Prop 3}}
and 
${\frak a}=\span_{\F}\{a_1,\ldots,a_r\}$ as in 
{\bf Lemma \ref{lema 2}}.
Let $\{v_1,\dots, v_r\}$  be the basis of $V$ used in {\bf Lemma \ref{lema 1}}
and let $h:\g\to\g_0$ be the map defined therein.
Then,

\begin{itemize}

\item[(i)] $\{a_1,\ldots,a_r\}$ y $\{h(v_1),\ldots,h(v_r)\}$ are both linearly independent sets.

\item[(ii)] $h(\a) \subset h(V)=\Ker(\alpha_0)$, 
$r=\dim_{\F}\Ker(\alpha_0)$ and $h\vert_V:V \rightarrow \Ker(\alpha_0)$ is bijective.

\item[(iii)] $B\left(\Im(\alpha_0),V\right)=\{0\}$ and $\alpha_0=\alpha_0 \circ h \circ \alpha_0$.

\item[(iv)] $E=\alpha_0 \circ (h\vert_{\g_0})$ is a projection ({\it ie\/,} $E^2=E$) 
and the decomposition of $\g_0$ associated to $E$ is $\g_0=\Im(\alpha_0) \oplus \a$.

\item[(v)] $F=h \circ \alpha_0$ is a projection ({\it ie\/,} $F^2=F$) and the decomposition of
$\g_0$ associated to $F$ is $\g_0=\Ker(\alpha_0) \oplus
\a^{\perp}$, where $\Ker(\alpha_0)=\Ker(F)$. Moreover, for each $x \in \g_0$,
\begin{equation}\label{l1}
x=F(x)+\sum_{i=1}^rB_0(a_i,x)\,h(v_i).
\end{equation}
\end{itemize}

In fact, the subspace $\a^{\perp}$ 
has the structure of a quadratic Lie algebra, such that,
$\g_0=\Ker(\alpha_0)\oplus \frak {a}^{\perp}$ is a central extension of
$\frak {a}^{\perp}$ by $\Ker(\alpha_0)$.
}
\end{Lemma}

\medskip
\begin{Cor}\label{a isotropico}{\sl
Let the hypotheses be as in {\bf Lemma \ref{lema 31}}.
If $V$ is isotropic, there is an invariant metric 
$\tilde{B}$ on $\g$
for which $\a\subset\g$ is an isotropic subspace.}
\end{Cor}
\medskip

\begin{Remark}{\rm
From now on, we may assume that ${\frak a}$ is an isotropic subspace of $\g$ which,
according to {\bf Lemma 2.11} in \cite{GSV}, is equivalent to state that $\Ker(h)={\frak a}$.
Having this in mind, we may now describe in better detail the structure
of the Hom-Lie algebra $(\g,\mu,\alpha^{\prime})$ given in {\bf Cor. \ref{corolario}}.
}
\end{Remark}
\medskip

\begin{Theorem}\label{teorema}{\sl
{\bf (A)} Let $(\g_0,[\,\cdot\,,\,\cdot\,]_0,B_0)$ be a finite dimensional 
non-Abelian quadratic Lie algebra over an 
algebraically closed field $\F$ of characteristic zero. 
Let $V$ be a finite dimensional vector space over $\F$ and let 
$(\g,[\,\cdot\,,\,\cdot\,])$ be a central extension of $\g_0$ by $V$
associated to a given the 2-cocycle $\theta:\g_0 \times \g_0 \to V$.
Suppose that there exists an invariant metric $B$ on $\g$
and let $(\g,\mu,\alpha^{\prime})$ be the Hom-Lie algebra
defined in {\bf Cor. \ref{corolario}}.
Let $h:\g \to \g_0$ be as in {\bf Lemma \ref{lema 1}}.
There is a bilinear map $\g\times\g\to\g$, 
denoted by $(x,y) \mapsto xy$, such that,
\begin{itemize}
\item[(i)] $\mu(x,y)=xy-yx$,
\smallskip
\item[(ii)] 
$B(xy,z)+B(y,xz)=0$,
\smallskip
\item[(iii)] $2\,xy=\mu(x,y)-[h(x),y]+[x,h(y)]$, for all $x,y,z \in \g$.
\smallskip
\item[(iv)] $x^{2}=0$ for all $x \in  \Ker(h) \cup \Im(\alpha^{\prime})$, 
where $\g=\Ker(h) \oplus \Im(\alpha^{\prime})$. 
In addition, for any $y=a+\alpha^{\prime}(x)\in \g$, 
with $a \in \Ker(h)$ and $x \in \g$, 
$$
y^2=(a+\alpha^{\prime}(x))^2=[a,x] \in \Im(\alpha^{\prime}).
$$
Therefore, $y^2y^2=0$ for all $y \in \g$.
\smallskip
\item[(v)] $xy=-yx$ for all $x,y \in \g$ if and only if 
$\theta=0$.
\end{itemize}

\medskip
{\bf (B)} Let $V$ be an isotropic ideal of $\g$ and assume $\theta$ is not a coboundary.
Let ${\frak G}=\{(\xi,x)\mid \xi \in \F,\,x \in \g\}$
and define the following multiplication map in it:
$$
\nu\left(\,(\xi,x)\,,\,(\eta,y)\,\right)=(\,\xi \eta+B(x,y)\,,\,\xi y+\eta x+x y\,).
$$
Then, ${\frak G}$ is an algebra with unit element $1_{\frak G}=(1,0)$
having no non-trivial two-sided ideals.
Let $\tilde{\nu}$ be the induced multiplication map on the quotient 
$\widetilde{\frak G} = {\frak G}/\Bbb F\,1_{\frak G}$ and
let $[\,\cdot\,,\,\cdot\,]_{\tilde{\nu}}$ 
be the commutator induced on $\widetilde{\frak G}\simeq\g$ by
the commutator defined by $\nu$ in ${\frak G}$:
$$
[\,(\xi,x)\,,\,(\eta,y)\,]_{\nu}=\nu\left(\,(\xi,x)\,,\,(\eta,y)\,\right)-\nu\left(\,(\eta,y)\,,\,(\xi,x)\,\right).
$$
If $\tilde{\alpha}$ is the linear map
induced on $\widetilde{\frak G}$ by the map 
$(\xi,x)\mapsto\left(\xi,\alpha^\prime(x)\right)$
on ${\frak G}$,
then $(\widetilde{\frak G},[\,\cdot\,,\,\cdot\,]_{\tilde{\nu}},\tilde{\alpha})$ 
is a Hom-Lie algebra isomorphic to $(\g,\mu,\alpha^\prime)$.
Moreover, if $\g_0$ is a nilpotent Lie algebra,
then $({\frak G},\nu)$ is simple\/.}
\end{Theorem}

\medskip
\begin{Remark}{\rm
Notice that the statement actually says that 
$({\frak G},{\nu})$ is a simple algebra, while the algebra 
$({\frak G},[\,\cdot\,,\,\cdot\,]_{\nu})$ defined by the commutator
is not simple
because the center $\F (1,0)$ is obviously an ideal.
Also note that the multiplication map $\tilde{\nu}$ induced by
$\nu$ on the quotient ${\frak G}/\Bbb F\,1_{\frak G}\simeq\g$
is neither symmetric nor skew-symmetric.
Finally, observe that
property {\bf (i)} in statement {\bf (A)} implies that the bilinear map
$(x,y) \mapsto xy$, somehow behaves like a connection adapted to $\mu$.
}
\end{Remark}

\medskip
\begin{proof}
The bilinear form $\g\times\g\ni (x,y)\mapsto xy\in\g$ can be defined
with the help of the invariant metric $B$ on $\g$
throughout the following identity that is to hold true for any triple
$x,y,z,\in\g$:
\begin{equation}\label{0}
2B(x y,z)=B(\mu(x,y),z)+B(\mu(z,x),y)+B(\mu(z,y),x).
\end{equation}
Once defined $xy$ this way, one may prove
statements {\bf (i)}, {\bf (ii)} and {\bf (iii)} in a straightforward way.
Notice in particular that {\bf (iii)} follows from the definition $\mu(x,y)=h([x,y])$,
{\bf Lemma \ref{lema 1}} and the invariance of $B$.
In particular, $xy=-yx$ for all $x,y \in \g$ if and only if $\mu(x,y)=2xy$,
which in turn implies that $B$ is invariant for $\mu$ and 
this is equivalent to the fact that $D_1=\cdots=D_r=0$; {\it ie\/,} $\theta=0$
(see {\bf Prop. \ref{invariancia hom-lie}}). Thus, {\bf (v)} also follows.

\medskip
To prove {\bf (iv)} we use {\bf Lemma \ref{lema 2}} and the fact that $V \subset C(\g)$, 
to conclude that, for all $x$, $y$ in $\g_0$,
$a$, $a^\prime$ in $\a=\Ker(h)$ and $u$, $v$ in $V$, 
$$
2\,\alpha^{\prime}(x+u)\,\alpha^{\prime}(y+v)=\alpha_0([x,y]_0),
\quad
\text{and,}
\quad
2\,a\,a^\prime = \mu(a,a^\prime).
$$
It follows that for any $x \in \Im(\alpha^{\prime}) \cup \Ker(h)$, $x^2=0$.
On the other hand, let $a \in \Ker(h)$ and $x \in \g_0$.
Using {\bf Lemma \ref{lema 2}} \eqref{tc1}, we obtain,
$$
\aligned
2a\,\alpha^{\prime}(x)&=\mu\left(a,\alpha^{\prime}(x)\right)
-[h(a),\alpha^{\prime}(x)]+[a,h(\alpha^{\prime}(x))]\\
\,&=h([a,k(x)])+[a,h(k(x))]=h(k([a,x]_0))+[a,x]\\
\,&=[a,x]_0+[a,x]=2[a,x]_0+\theta(a,x).
\endaligned
$$
Similarly,
$$
\aligned
2\alpha^\prime(x)\,a&=\mu(\alpha^{\prime}(x),a)-[h(\alpha^{\prime}(x)),a]+[\alpha^{\prime}(x),h(a)]\\
\,&=h([k(x),a])-[h(k(x)),a]\\
\,&=h(k([x,a]_0))-[x,a]=[x,a]_0-[x,a]\\
\,&=-\theta(x,a)=\theta(a,x).
\endaligned
$$
Combining these two expressions, we obtain,
$$
2\,\left(\,a\,\alpha^{\prime}(x)+\alpha^{\prime}(x)\,a\,\right)=2[a,x]_0+2\theta(a,x)=
2\,[a,x].
$$
Since $\alpha^{\prime}(x)^2=a^2=0$, we finally obtain:
$$
\left(a+\alpha^{\prime}(x)\right)^2
\!\!=[a,x] \in \Im(k)=\Im(\alpha^{\prime}),\ \ \text{for all\ }a\in \Ker(h),\ \text{and}\  x \in \g_0.
$$
Therefore, $(a+\alpha^{\prime}(x))^4=0$. 

\medskip
We shall now prove statement {\bf (B)}.
Assume $V$ is isotropic and $\theta$ is not a coboundary.
In particular, this implies that $\Ker(\alpha_0)=C(\g_0)$ and $\Im(\alpha_0)=[\g_0,\g_0]_0$
(see {\bf Prop. \ref{perfect hom-Lie algebra}.(iii)}). 
Let 
${\frak G}$ be as in the statement and
define equality, addition and multiplication by scalars
in ${\frak G}$ in the obvious manner.
Let $\nu$ and $\tilde{\nu}$ be the multiplications given in the statement.
If $[\!\![\,\cdot\,]\!\!]:{\frak G}\to\widetilde{\frak G}={\frak G}/{\Bbb F}1_{{\frak G}}$
is the projection map onto the quotient, then,
$$
\tilde{\nu}\left(\,[\!\![(\xi,x)]\!\!],[\!\![(\eta,y)]\!\!]\,\right)=
[\!\![ \,\nu\left((\xi,x)\,,\,(\eta,y)\right)\, ]\!\!]
$$
Since $\theta$ is not a coboundary,
{\bf Lemma \ref{lema 2}.(iv)} makes us conclude that this multiplication is neither
commutative (symmetric) nor anti-commutaive (skew-symmetric).

\medskip
We claim that any non-trivial two-sided ideal in ${\frak G}$
contains the unit element $1_{{\frak G}}=(1,0)$.
Indeed, let $I \neq \{0\}$ be a two-sided ideal of ${\frak G}$. 
We shall see that $I$ has an element of the form $(\xi,\alpha^{\prime}(x))$.
Let $(\xi,x) \in I$ with $x \in \g_0-\{0\}$ 
and let $y \in \g_0$ be arbitrary.
Using {\bf Lemma \ref{lema 1}.(i)-(ii)}, we deduce that,
$$
\aligned
\nu\left(\,(\xi,x)\,,\,(0,\alpha^{\prime}(y))\,\right)&=
\left(B(x,\alpha^{\prime}(y)),\xi \alpha^{\prime}(y)+x\alpha^{\prime}(y)\right)\\
\,&=\left(B_0(x,y),\xi \alpha^{\prime}(y)+x\alpha^{\prime}(y)\right) \in I,
\endaligned
$$
where we have used the fact that 
$B(x,\alpha^{\prime}(y))=B(x,k(y))=B_0(x,y)$ for all $x,y \in \g_0$.
Now, let $y \in \g_0$ be such that $B_0(x,y) \neq 0$. If $\xi \alpha^{\prime}(y)+x\alpha^{\prime}(y)=0$, 
we are done because in that case $B_0(x,y)(1,0) \in I$.
If $\xi \alpha^{\prime}(y)+x\alpha^{\prime}(y) \neq 0$, then 
$(B_0(x,y),\xi \alpha^{\prime}(y)+x\alpha^{\prime}(y)) \in I$, 
with $B_0(x,y) \neq 0$ and $\xi \alpha^{\prime}(y)+x\alpha^{\prime}(y) \neq 0$, and
$$
2\,x\,\alpha^{\prime}(y)=[x,y]-\alpha^{\prime}([h(x),y])+[x,y].
$$
Now, $[x,y] \in [\g_0,\g_0]=\Im(\alpha_0) \subset V^{\perp}=\Im(k)=\Im(\alpha^{\prime})$
(this follows from {\bf Prop. \ref{perfect hom-Lie algebra}.(iii)} and {\bf Lemma \ref{lema 31}.(iii)}).
On the other hand,
$\alpha^{\prime}([h(x),y])+[x,y] \in V^{\perp}=\Im(k)=\Im(\alpha^{\prime})$. 
Therefore, $2\,x\,\alpha^{\prime}(y) \in \Im(\alpha^{\prime})$, which proves our claim.

\medskip 
Let $(\xi,\alpha^{\prime}(x)) \in I-\{0\}$ with $\xi \alpha^{\prime}(x) \neq 0$.
Since $\alpha^{\prime}(x)^2=0$, 
$$
\nu\left(\,(\xi,\alpha^{\prime}(x)),(-\xi,\alpha^{\prime}(x))\right)
=(-\xi^{2}+B\left(\alpha^{\prime}(x)\,,\,\alpha^{\prime}(x)\right),0\,)\in I.
$$
If $-\xi^2+B(\alpha^{\prime}(x),\alpha^{\prime}(x)) \neq 0$, then $(1,0) \in I$.
Now suppose that,
\begin{equation}\label{clubsuit}
\left(\xi,\alpha^{\prime}(x)\right) \in I,\quad\text{and}\quad
-\xi^2+B\left(\alpha^{\prime}(x),\alpha^{\prime}(x)\right)=0.
\end{equation}
Then, for any $(\eta,\alpha^{\prime}(y)) \in {\frak G}$,
with arbitrary $\eta\in\Bbb F$ and $y \in \g$, we have,
$$
\aligned
\frac{1}{2}
&
\left(\nu
(
{\left(\xi,\alpha^{\prime}(x)),(\eta,\alpha^{\prime}(y)\right)}
)
+
\nu
(
{\left(\eta,\alpha^{\prime}(y)),(\xi,\alpha^{\prime}(x)\right)}
)
\right)
\\
&=
\left(
\xi \eta + B(\alpha^{\prime}(x),\alpha^{\prime}(y)),\xi \alpha^{\prime}(y)+\eta \alpha^{\prime}(x)
\right).
\endaligned
$$
Under the assumption \eqref{clubsuit}, we obtain,
$$
\xi \eta+B\left(\alpha^{\prime}(x),\alpha^{\prime}(y)\right)
=B\left(\,
\xi \alpha^{\prime}(y)+\eta \alpha^{\prime}(x)\,,\,\xi \alpha^{\prime}(y)+\eta \alpha^{\prime}(x)
\,\right),
$$
for any $\eta\in\Bbb F$ and $y \in \g$.
Whence,
$B(\alpha^{\prime}(x),\alpha^{\prime}(y))^2=\xi^2B(\alpha^{\prime}(x),\alpha^{\prime}(y))$, 
for all $y \in \g$. 
By polarization of this identity we get, for any $y,z\in\g$,
\begin{equation}\label{linealizacion}
B\left(\alpha^{\prime}(x),\alpha^{\prime}(y)\right)
B\left(\alpha^{\prime}(x),\alpha^{\prime}(z)\right)
=\xi^2B\left(\alpha^{\prime}(y),\alpha^{\prime}(z)\right).
\end{equation}
Since $\alpha^\prime$ is self-adjoint under $B$,
the non-degeneracy of $B$ leads to,
\begin{equation}\label{p1}
B\left(\alpha^{\prime}(x),\alpha^{\prime}(y)\right)\,
\alpha^\prime\left(\alpha^\prime(x)\right)
=\xi^2\,\alpha^\prime\left(\alpha^\prime(y)\right),
\quad \text{for any\ } y \in \g.
\end{equation}
Observe that for any $y \in \g_0$:
\begin{equation}
\label{p2}
\alpha^\prime\left(\alpha^\prime(y)\right)
=\alpha^\prime\left(k(y)\right)
=\alpha^\prime\left(\alpha_0(y)\right)
=k\left(\alpha_0(y)\right),
\end{equation}
and,
\begin{equation}\label{p2-2}
B\left(\alpha^{\prime}(x),\alpha^{\prime}(y)\right)
=B\left(k(x),k(y)\right)
=B\left(k(x),\alpha_0(y)\right)
=B_0\left(x,\alpha_0(y)\right).
\end{equation}
After substitution
of \eqref{p2} and \eqref{p2-2} in \eqref{p1}, 
we get,
$$
B\left(x,\alpha_0(y)\right)\,k\left(\alpha_0(x)\right)=\xi^2k(\alpha_0(y)),\quad \text{for all\ }y \in \g_0.
$$
We may now 
use  \textbf{Lemma \ref{lema 1}} and
apply $h$ on both sides to obtain, 
\begin{equation}\label{last1}
B(x,\alpha_0(y))\alpha_0(x)=\xi^2\alpha_0(y),\quad \text{for all\ }y \in \g_0.
\end{equation}
In particular, for $\mu(x,y) \in \g_0$, 
$\alpha_0(\mu(x,y))=[x,y]_0$
(see {\bf Lemma \ref{lema 2}.(v)}).
Applying this in \eqref{last1} we get,
$$
\aligned
\xi^2[x,y]&=\xi^2\alpha_0(\mu(x,y))=B_0(x,\alpha_0(\mu(x,y)))\,\alpha_0(x)\\
\, &=B_0(x,[x,y]_0)\,\alpha_0(x)=0,\ \ \text{for all\ }y \in \g_0.\\
\endaligned
$$
Then $x \in C(\g_0)=\Ker(\alpha_0)$. From \eqref{p2}, 
we have $\alpha^{\prime}\left(\alpha^{\prime}(x)\right)=0$. 
Now use in \eqref{linealizacion} the fact that 
$\alpha^{\prime}$ is $B$-symmetric, to obtain,
$B(\alpha^{\prime}(y),\alpha^{\prime}(z))=0$, for all $y,z \in \g_0$. But,
$$
B\left(\alpha^{\prime}(y),\alpha^{\prime}(z)\right)
=B\left(k(y),k(z)\right)
=B\left(\alpha_0(y),k(z)\right)
=B_0\left(\alpha_0(y),z\right)=0,
$$
thus implying that $\alpha_0=0$.
Thus, $\g_0$ is Abelian 
(see {\bf Cor. \ref{corolario zero}}), which is a contradiction.
Therefore, there must be an element of the form $(\xi,\alpha^{\prime}(x)) \in I$
with $-\xi^2+B\left(\alpha^{\prime}(x),\alpha^{\prime}(x)\right) \ne 0$,
which implies that $(1,0)$ belongs to $I$. 

\medskip
Now, let $[\,\cdot\,,\,\cdot\,]_{\nu}$ be the commutator in ${\frak G}$ defined
in terms of $\nu$. It is clear that for any $x,y\in\g$ and any scalars $\xi,\eta\in\Bbb F$,
$$
\left[(\xi,x),(\eta,y)\right]_{\nu} =
\nu\left((\xi,x),(\eta,y)\right)-
\nu\left((\eta,y),(\xi,x)\right)
=(0,\mu(x,y)).
$$
Now let $[\,\cdot\,,\,\cdot\,]_{\tilde{\nu}}:\widetilde{\frak G}\times\widetilde{\frak G}\to\widetilde{\frak G}$
be the induced skew-symmetric multiplication on $\widetilde{\frak G}=
{\frak G}/\Bbb F1_{\frak G}$
by the induced multiplication $\tilde{\nu}$ on this quotient.
Clearly, $\widetilde{\frak G}\simeq\g$, and,
$$
\left[\,[\!\![\,(0,x)\,]\!\!]\,,\,[\!\![\,(0,y)\,]\!\!]\,\right]_{\tilde{\nu}} 
\ \longleftrightarrow\  \mu(x,y),\quad\text{for all\ }x,y\in\g.
$$
If $\tilde{\alpha}:\widetilde{\frak G} \to \widetilde{\frak G}$ is
the linear map induced on the quotient by,
$$
(\xi,x)\ \mapsto\ (\xi,\alpha(x)),\quad \text{for all\ }\xi \in \F\ \text{and\ } x \in \g,
$$ 
it follows that $(\widetilde{\frak G},[\,\cdot\,,\,\cdot\,]_{\tilde{\nu}},\tilde{\alpha})$ is a Hom-Lie algebra
isomorphic to $(\g,\mu,\alpha^\prime)$.

\medskip
Finally, we shall prove the if $\g_0$ is a nilpotent Lie algebra, 
then ${\frak G}$ is simple for ${\nu}$.
Thus, assume $\g_0$ is nilpotent for $[\,\cdot\,,\,\cdot\,]_0$.
Let $I \neq \{0\}$ be
a right ideal of ${\frak G}$.
Proceeding as before, we may assume that there is an element
$(\xi,\alpha^{\prime}(x)) \in I$, such that $\xi \alpha^{\prime}(x) \neq 0$.
Make the assumption \eqref{clubsuit}, and let $y\in\g_0$ be arbitrary.
Then,
$$
\nu
\left(
(\xi,\alpha^{\prime}(x)),(0,\alpha^{\prime}(y))
\right)
=(B(\alpha^{\prime}(x),\alpha^{\prime}(y)),\xi \alpha^{\prime}(y)+\frac{1}{2}\alpha_0([x,y]_0)) \in I.
$$
Since $\Im(\alpha_0) \subset V^{\perp}=\Im(k)=\Im(\alpha^{\prime})$, (see {\bf Lemma \ref{lema 31}.(iii)}), the assumption \eqref{clubsuit} leads to,
$$
\aligned
B\left(\alpha^{\prime}(x),\alpha^{\prime}(y)\right)^2&=\xi^2B\left(\alpha^{\prime}(y),\alpha^{\prime}(y)\right)+
\displaystyle{\frac{1}{4}}\,B\left(\alpha_0([x,y]_0),\alpha_0([x,y]_0)\right),\\
\,&=\xi^2\,B\left(\alpha^{\prime}(y),\alpha^{\prime}(y)\right)+
\displaystyle{\frac{1}{4}}\,B\left(\alpha^{\prime}([x,y]),\alpha^{\prime}([x,y])\right),\\
\,&=\xi^2B\left(\alpha^{\prime}(y),\alpha^{\prime}(y)\right)+
\displaystyle{\frac{1}{4}}\,B\left([\alpha^{\prime}(x),y],[\alpha^{\prime}(x),y]\right),
\endaligned
$$
which holds true for any $y\in\g_0$.
Observe that we have used the fact that
$B(\alpha_0(x),\alpha_0(y))=B(\alpha(x),\alpha_0(y))=B(\alpha(x),\alpha(y))$ and $\alpha^{\prime}([x,y])=[\alpha^{\prime}(x),y]$, which follows from
$\Im(\alpha_0) \subset V^{\perp}=\Im(\alpha^{\prime})$ and $\alpha^{\prime} \in \Gamma_B(\g)$ (see {\bf Cor. \ref{corolario}}).
Applying the last identity to $y+z\in\g_0$ instead of $y\in\g_0$, we get,
for any $y, z\in\g_0$, 
\begin{equation}\label{linearizacion 2}
{
\aligned
B\left(\alpha^{\prime}(x),\alpha^{\prime}(y)\right)
\,
B\left(\alpha^{\prime}(x),\alpha^{\prime}(z)\right)
& =
\xi^2B\left(\alpha^{\prime}(y),\alpha^{\prime}(z)\right)
\\
& \quad +
\displaystyle{\frac{1}{4}}\,B\left([\alpha^{\prime}(x),y],[\alpha^{\prime}(x),z]\right).
\endaligned
}
\end{equation}
Using again the fact that $B$ is an invariant metric
in $\g$ with $\alpha^\prime\in\Gamma_B(\g)$, we get for any $y\in\g_0$,
\begin{equation}\label{linearizacion 3}
B\left(\alpha^{\prime}(x),\alpha^{\prime}(y)\right)
\,
\alpha^{\prime}\left(\alpha^{\prime}(x)\right)
=\xi^2\alpha^{\prime}\left(\alpha^{\prime}(y)\right)
-
\displaystyle{\frac{1}{4}}\,
\alpha^{\prime}\left([x,[x,\alpha^{\prime}(y)]]\right).
\end{equation}
Now put in this expression $\mu(x,y)\in\g_0$ instead of $y$
and use \eqref{delta producto-1} to make the simplification
$\alpha^{\prime}\left(\mu(x,y)\right)=k\left(\mu(x,y)\right)=[x,y]$
and conclude that,
\begin{equation}\label{linearizacion 3-2}
B\left(\alpha^{\prime}(x),\mu(x,y)\right)
= B\left(x,\alpha^{\prime}(\mu(x,y))\right)
= B(x,[x,y])=0.
\end{equation}
Thus, substituting \eqref{linearizacion 3-2} in \eqref{linearizacion 3} leads to:
$$
\aligned
\xi^2 \alpha^{\prime}([x,y])-\displaystyle{\frac{1}{4}}\,
\alpha^{\prime}
\left(
{\left[x,{\left[x,[x,y]\right]}\right]}
\right)
&=0,\quad\text{which is equivalent to}
\\
\xi^2 k([x,y])-\displaystyle{\frac{1}{4}}\,k\left(
{\left[x,{\left[x,[x,y]_{0}\right]_0}\right]_0}
\right)
&=0,
\quad\text{ for all }\, y \in \g_0.
\endaligned
$$
Since $k$ is injective, it follows that 
$4\, \xi^2\ad_0(x)(y)=\ad_0(x)^2(\ad_{\g_0}(x)(y))$, for all $y \in \g_0$. 
Thus, if $\g_0$ is a nilpotent Lie algebra and $\xi \neq 0$, 
then $\ad_{\g_0}(x)(y)=0$, for all $y \in \g_0$, 
which implies that
$x \in C(\g_0)=\Ker(\alpha_0)$
and by \eqref{p2}, 
$\alpha^{\prime}\left(\alpha^{\prime}(x)\right)=0$.
From \eqref{linearizacion 2}, this implies that 
$B\left(\alpha^{\prime}(y),\alpha^{\prime}(z)\right)=0$, for all $y,z \in \g_0$. 
Then, observe from \eqref{p2} that this would
imply $\alpha_0=0$ making $\g_0$ Abelian,
which is a contradiction (see {\bf Cor. \ref{corolario zero}}). 
Therefore, ${\frak G}$ 
can have no non-trivial right-ideals
if $\g_0$ is a nilpotent 
Lie algebra. The proof for left-ideals is similar.
\end{proof}

\medskip
\section{An illustrative non-trivial example}

Let $(\g_0,[\,\cdot\,,\,\cdot\,]_0,B_0)$ be a 6-dimensional $2$-step nilpotent
quadratic Lie algebra, where $\g_0=\a \oplus {\frak b}$,
$\a=\operatorname{Span}_{\F}\{a_1,a_2,a_3\}$ and
${\frak b}=\operatorname{Span}_{\F}\{b_1,b_2,b_3\}$.
Let $\sigma$ be the cyclic permutation $\sigma=(1\,2\,3)$.
The Lie bracket $[\,\cdot\,,\cdot\,]_0$ is given by,
$$
[\g_0,{\frak b}]_0=\{0\},
\quad\text{and}\quad 
[a_i,a_{\sigma(i)}]_0=b_{\sigma^2(i)},\quad (1 \le i \le 3).
$$
The invariant metric 
$B_0$ is defined by making ${\frak a}$ and ${\frak b}$
totally isotropic and choosing an invertible $3\times 3$
matrix $(\beta_{jk})$ such that
$B_0(a_j,b_k)=\beta_{jk}$.
We define 
the $B_0$-skew-symmetric derivations
$D_1,D_2,D_3\in\Der\g_0$, through,
$$
\aligned
D_i(a_{\sigma(i)})& =-D_{\sigma(i)}(a_i)=a_{\sigma^2(i)}+b_{\sigma^2(i)},
\\
D_{i}(b_{\sigma(i)}) & =-D_{\sigma(i)}(b_i)=b_{\sigma^2(i)},
\\
D_i(a_i)& =D_i(b_i)=0,
\endaligned
$$
holding true for $1\le i\le 3$.
Observe that the derivations $D_1,D_2,D_3$ are not inner. 
Were they inner, their image
would lie in $\operatorname{Span}_{\F}\{b_1,b_2,b_3\}
=[\g_0,\g_0]_0$, which is not the case.

\medskip
Let $V=\operatorname{Span}_{\F}\{v_1,v_2,v_3 \}$ be a $3$-dimensional vector space. 
Let $\g=\g_0 \oplus V$, and define 
$\theta:\g_0 \times \g_0 \rightarrow V$ by, $\theta(x,y)=\sum_{i=1}^3B_0(D_i(x),y)\,v_i$, for all $x,y \in \g_0$. Thus,
$$
\theta(a_i,a_{\sigma(i)})=v_{\sigma^2(i)}
$$
We use $\theta$ to define the Lie bracket on $\g$ as in \eqref{tc1}; that is,
$$
[a_i,a_{\sigma(i)}] = b_{\sigma^2(i)}+v_{\sigma^2(i)}
$$
It follows that the $9$-dimensional Lie algebra $\g$ thus defined is
a $3$-step nilpotent, quadratic Lie algebra. 
Now,  the symmetric bilinear form $B:\g \times \g \rightarrow \F$,
given by, $B(b_i,b_j)=B(a_i,v_j)=\delta_{ij}$, and
$B(b_i,a_j+v_k)=B(a_i,a_j)=B(v_i,v_j)=0$,
for all $1\le i,j,k\le 3$, yields an invariant metric on $\g$. 
Clearly, $\g$ has the vector space decomposition, 
$\g=\a \oplus {\frak b} \oplus V$, where both, $\a$ and $V$, are isotropic
subspaces under $B$, and ${{\frak b}}^{\perp}=\a\oplus V$. 
\medskip

We shall now describe the Hom-Lie algebras
$(\g_0,\mu_0,\alpha_0)$, $(\g,\mu,\alpha)$
and $(\g,\mu,\alpha^\prime)$ of 
{\bf Prop. \ref{Prop 3}}
and
{\bf Cor. \ref{corolario}}.
Note first that the linear map $h:\g \to \g_0$ 
of {\bf Lemma \ref{lema 1}} defined in \eqref{maps} is given by:
$$
h(a_{j})=0,\quad h(b_j)=a_{j},\quad h(v_j)=b_j,\quad \text{for all\ } 1 \leq j \leq 3.
$$
Now, according to {\bf Prop. \ref{Prop 3}},
the map $\mu:\g\times\g\to\g$ is defined thorugh $h$
and the Lie bracket $[\,\cdot\,,\,\cdot\,]$ in $\g$ as,
$\mu(x,y)=h([x,y])$ for all $x$ and $y$ in $\g$. Since $V \subset C(\g)$, then
for any $x,y\in\g_0$ and $u,v\in V$,
$$
\mu(x+u,y+v)=\mu_0(x,y)
$$
Observe that ${\frak b}=C(\g_0)$ is $B$-isotropic and 
$D_j({\frak b}) \subset {\frak b}$ for all $j$. Thus, 
$\mu({\frak b},{\frak b})=h([{\frak b},{\frak b}])=\{0\}$. Therefore, for all $1 \leq j,k \leq 3$ and for all $x \in \g_0$, the product $\mu$ is given by:
$$
\aligned
\mu(a_j,a_{\sigma(j)})&=h([a_j,a_{\sigma(j)}])=D_j(a_{\sigma(j)})=a_{\sigma^2(j)}+b_{\sigma^2(j)},\\
\mu(a_j,b_{\sigma(j)})&=h([a_j,b_{\sigma(j)}])=D_j(b_{\sigma(j)})=b_{\sigma^2(j)},\\
\mu(b_j,b_k)&=0,\\
\mu(v_j,x) & =h([v_j,x])=0.
\endaligned
$$
The linear map $\alpha_0:\g_0\to\g_0$ of {\bf Prop. \ref{Prop 3}}
that completes the description of the 
$6$-dimensional Hom-Lie algebra $(\g_0,\mu_0,\alpha_0)$ is given by,
$$
\alpha_0(a_j)=b_j, \quad \alpha_0(b_j)=0,\quad 1 \leq j \leq 3.
$$
In particular, $k:\g_0 \to \g$, which according to
{\bf Prop. \ref{Prop 3}} satisfies 
$k(x)=\alpha_0(x)+\sum_{i=1}^3B_0(a_i,x)v_i$, for all $x \in \g_0$,
is given by,
$$
k(a_j) = b_j,\quad k(b_j)=v_j,\quad 1 \le j \le 3.
$$
The invariance of $B_0$ under $\mu_0$,
is obtained from the $B_0$-skew-symmetry of the derivations 
$D_1$, $D_2$ and $D_3$ and from the fact that ${\frak b}$ is $\mu$-Abelian.
Thus, for any $x$ and $y$ in $\g_0$ we have,
$$
\aligned
B_0(\mu(a_j,x),y)& =B_0(D_j(x),y)\\
&=-B_0(x,D_j(y))\\
&=-B_0(x,\mu(a_j,x)).
\endaligned
$$
On the other hand, the maps $\alpha:\g\to\g$ and $\alpha^\prime:\g\to\g$
that define the $9$-dimensional Hom-Lie algebras $(\g,\mu,\alpha)$
and $(\g,\mu,\alpha^\prime)$
are respectively given by,
$\alpha(x+v)=\alpha_0(x)+v$ and
$\alpha^{\prime}(x+v)=k(x)$, for all $x \in \g_0$ and all $v\in V$.
These yield,
$$
\alpha(a_j) =b_j,\quad \alpha(b_j)=0,\quad \alpha(v_j)=v_j,\quad 1 \le j \le 3.
$$
Similarly
$$
\alpha^\prime(a_j) =b_j,\quad \alpha^{\prime}(b_j)=v_j,\quad \alpha^{\prime}(v_j)=0,\quad 1 \le j \le 3.
$$
Clearly, these two linear maps are different because 
$\Ker(\alpha^{\prime})=V$, while $\alpha\vert_V=\operatorname{Id}_V$.

\medskip

We shall finally describe the
bilinear {\it connection map\/} $(x,y) \mapsto xy$ 
in terms of the basis $\{\,a_j,b_k,v_{\ell}\mid\,1 \le j,k, \ell \le 3\,\}$ of $\g$. 
By {\bf Thm. \ref{teorema}}, we know that:
$$
2xy=\mu(x,y)-[h(x),y]+[x,h(y)],\quad \text{for all\ }x,y \in \g.
$$
Since $\mu(a_j,\cdot)|_{\g_0}=D_j$ and $h(a_j)=0$, for all $j$, we obtain:
$$
2a_jx=D_j(x)+[a_j,h(x)],\quad \text{for all\ }x \in \g.
$$
In particular, $2a_ja_k=D_j(a_k)=-D_k(a_j)=-2a_ka_j$ for all $j,k$. 
This agrees with the fact that the connection is anti-commutative
on the elements of $\Ker(h)={\frak a}$. On the other hand,
$$
2a_1b_2=2b_3+v_3,\quad 2a_2b_3=2b_1+v_1.\quad 2a_3b_1=2b_2+v_2.
$$
Analogously,
$$
2b_1a_2=v_3,\quad 2b_2a_3=v_1,\quad 2b_3a_1=v_2.
$$
This shows that the connection is neither commutative nor anti-commutative. 
We shall now compute the connection on ${\frak b}$. 
Since ${\frak b}$ is a $\mu$-Abelian subspace of $\g_0$
which is
invariant under $D_j$ for all $j$, we get,
$$
2b_jb_k=-[h(b_j),b_k]+[b_j,h(b_k)]=-[a_j,b_k]-[a_k,b_j],\quad 1\le j,k\le 3.
$$
Now, using the $B_0$-skew-symmetry of the derivations $D_j$
it is not difficult to prove that in fact $b_jb_k=0$ for all $1\le j,k\le 3$.

\medskip
Finally, observe that in this example $\g_0$ is a Nilpotent Lie algebra 
and that it satisfies the hypotheses
of {\bf Thm. \ref{Teorema Cartan-Killing}}. 
Then, the symmetric and $\mu_0$-invariant bilinear form $K_0$ 
for $(\g_0,\mu_0,\alpha_0)$ is zero.

\medskip
\section*{Acknowledgements}
The authors acknowledge the support received
through CONACYT Grant $\#$ A1-S-45886. 
The author RGD would also like to thank the support 
provided by CONACYT post-doctoral fellowship 000153.
Finally, GS acknowledges the
support provided by PROMEP grant UASLP-CA-228
and ASV acknowledges the support given by MB1411.

\end{document}